\documentclass[a4paper,10pt]{article}
\usepackage[utf8]{inputenc}
\usepackage{amsmath,amsthm, amsfonts, amssymb, graphicx, bm, color, ulem}
\usepackage{tikz}
\usepackage[top=2.5cm, bottom=2.5cm, left=2.5cm, right=2.5cm]{geometry}
 
\def\minus{%
{\resizebox{ 2px}{!}{\rm{-}}}
}
\def\plus{
{\resizebox{4px}{!}{\bf{+}}}
}

 \newtheorem{theorem}{Theorem}[section]
 \newtheorem{lemma}[theorem]{Lemma}
 \newtheorem{corollary}[theorem]{Corollary}
 \newtheorem{proposition}[theorem]{Proposition}
 \newtheorem{definition}[theorem]{Definition}

 \newenvironment{remark}{ \refstepcounter{theorem}{\bf Remark \thetheorem~}}{\hfill$\square$\medskip}
 
 \newenvironment{example}{ \refstepcounter{theorem}{\bf Example \thetheorem~}}{\hfill$\square$\medskip}
 
\newcommand{\Sym}{\Phi_{\minus\theta}}
\newcommand\samethanks[1][\value{footnote}]{\footnotemark[#1]}

\title{The omnidirectional trace in $H^1(\Omega)$}
\author{R. Eymard\thanks{Universit\'e Gustave Eiffel, LAMA, (UMR 8050), UPEM, UPEC, CNRS, F-77454, Marne-la-Vall\'ee (France)\newline
{\tt robert.eymard, david.maltese, lucas.oger@univ-eiffel.fr}},
T. Gallou\"et\thanks{I2M UMR 7373, Aix-Marseille Universit\'e, CNRS, Ecole Centrale de Marseille,
F-13453 Marseille (France)\newline
{\tt thierry.gallouet@anciens.univ-amu.fr}},
D. Maltese\samethanks[1] \ and L. Oger\samethanks[1]}

\begin{document}
\maketitle

\begin{abstract}
We first prove that all the functions in $L^2$ whose directional derivative is in $L^2$ have a directional trace on the boundary of any open bounded domain, without assumptions on its regularity. This enables us to define the omnidirectional trace of the elements of the Sobolev space $H^1(\Omega)$ for which there exists a function on the boundary that is almost everywhere equal, with respect to the directional measure, to the directional trace, regardless of the direction. The set of all these elements of $H^1(\Omega)$, denoted by $H^1_{\rm tr}(\Omega)$, is shown to be closed, and to always contain the closure in $H^1(\Omega)$ of the set $C^0(\overline{\Omega})\cap H^1(\Omega)$ (it is always equal to this set in the 1D case, and can be strictly greater in higher dimensions). The omnidirectional trace always satisfies an integration-by-parts formula, which combines the values of the trace on opposite points of the boundary. Examples show  that this notion enables the resolution of variational problems involving the values at the boundary of the domain.
\end{abstract}

{\bf Keywords:} trace operator, Sobolev space, integration-by-parts, directional trace, integration for a family of measures, Cantor sets.

{\bf AMS classification:} 46E35, 46E40, 28A75, 35A23, 26B20

\section{Introduction}

We consider, for an integer $d \ge 1$, a bounded open domain $\Omega\subset {\mathbb R}^d$, denoting by $\partial\Omega = \overline{\Omega} \setminus \Omega$ its boundary. For a given function $u \in H^1(\Omega)$ -- the set of elements of $L^2(\Omega)$ whose gradient in the distribution sense is an element of $L^2(\Omega)^d$ -- defining a trace of $u$ on the boundary of $\Omega$ is an important problem in the theory of partial differential equations. Indeed, the elements of $H^1(\Omega)$ may have no continuous representative on $\overline{\Omega}$ if $d \ge 2$, and the use of a trace operator is indispensable for writing integration-by-parts formulas that underpin the weak formulation of such equations.

\medskip

The first observation is that, if some part of the boundary of $\Omega$ is accessible by two sides (this is the case of domains with cracks), defining a single value for $u \in H^1(\Omega)$ on this part of the boundary is not possible, unless $u$ is continuous across it. This means that we cannot expect that, for a general open bounded domain $\Omega \subset {\mathbb R}^d$, the domain of such a trace operator be the whole space $H^1(\Omega)$. The goal is therefore to define the trace operator on a subset of $H^1(\Omega)$ that coincides with the whole space when additional hypotheses are imposed on $\Omega$.

\medskip

{\it From the trace on Lipschitz domains to the ``approximative'' trace.}

\medskip

Let us recall that the notion of trace is standardly introduced in the case of domains with Lipschitz regularity. The existence of this trace is obtained through smooth local charts which transform the computations at the boundary of the domain into computations on $\mathbb{R}^{d-1}$ viewed as the half-space boundary $\partial(\mathbb{R}_+ \times \mathbb{R}^{d-1})$ (see the seminal paper \cite{Gagliardo1957}, and see \cite[Theory of traces]{brezis} for the introduction of local charts). The trace of any function $u \in H^1(\Omega)$ can then be defined as the limit, in $L^2(\partial\Omega, \mathcal{H}^{d-1})$, of the restrictions to $\partial\Omega$ of smooth functions converging to $u$ in $H^1(\Omega)$. Here $\mathcal{H}^{d-1}$ denotes the $(d-1)$-dimensional Hausdorff measure on $\mathbb{R}^d$ (see \cite[Definition 2.1]{evans2015meas}).

\medskip

For more general domains, this procedure was extended in \cite{ae2011dirtoneu, are2024perron, are2003lap, sauter2020uniq}. In these works, the set of the so-called approximative traces of a given $u \in H^1(\Omega)$ is that of all $\varphi \in L^2(\partial \Omega, \mathcal{H}^{d-1})$ such that there exists a sequence $(u_n)_{n \ge 0}$ in $H^1(\Omega) \cap C^0(\overline{\Omega})$ such that the sequence $(u_n)_{n \ge 0}$ converges to $u$ in $H^1(\Omega)$ and  $({u_n}_{|_{\partial\Omega}})_{n \ge 0}$ converges to $\varphi$ in $L^2(\partial \Omega,\mathcal{H}^{d-1})$. This definition is used in \cite{ae2011dirtoneu} with the aim of studying the Dirichlet-to-Neumann operator.

\medskip

Let us observe that this definition opens a series of restrictions and questions. It requires  the hypothesis that $\mathcal{H}^{d-1}(\partial\Omega) < +\infty$, which excludes the case of domains with fractal boundaries. Denoting by $\widetilde{H}^1(\Omega)$ the closure of $H^1(\Omega)\cap C^0(\overline{\Omega})$ in $H^1(\Omega)$, the set of the approximative traces is empty for all $u \in H^1(\Omega) \setminus \widetilde{H}^1(\Omega)$, and may contain infinitely many elements. Indeed, the uniqueness of an approximative trace is discussed in \cite{are2024perron} and is intensively studied in \cite{sauter2020uniq}, considering functions on the boundary only locally integrable for the measure $\mathcal{H}^{d-1}$ (counterexamples for this uniqueness are given in \cite[Example 4.3]{are2003lap} or \cite[end of Section 3]{bucur2010var}). As noticed in these works, this uniqueness property is related to the fact that the boundary of the domain may contain points that are not connected to $\Omega$, and are therefore not supporting the trace of the elements of $H^1(\Omega)$. We can remark that similar observations are made in \cite{shvartsman}, in order to define the trace of functions in the Sobolev space $W^{1,p}(\Omega)$ with $p > d$ on the boundary of general domains, introducing the notion of ``$\alpha$-accessible points''.

\medskip

We also observe that the $(d-1)$-dimensional Hausdorff measure on $\mathbb{R}^d$ appears in the integration-by-parts formula in the case of Lipschitz domains, and that generalisations can be obtained in the case of some more general domains (see finite perimeter sets and Gauss-Green theorem \cite[Section 5.8]{evans2015meas}). Nevertheless, the loss of uniqueness of the approximative trace when passing to the limit of sequence of continuous functions in the space $L^2(\partial \Omega, \mathcal{H}^{d-1})$ suggests that this measure becomes inappropriate for defining the topology on a set of functions on the boundary of more general domains, linked with an integration-by-parts formula. 
 
\medskip

{\it From the directional trace to the omnidirectional trace.}

\medskip

The approach we follow in this paper is entirely different. We first consider, for all directions $\theta \in \mathbb{S}^{d-1}$, where $\mathbb{S}^{d-1}$ denotes the unit sphere in $\mathbb{R}^d$, the space $W_\theta(\Omega)$ (defined in Section \ref{sec:traceund}) of all functions of $L^2(\Omega)$ whose derivative in the direction $\theta$, denoted by $\partial_\theta u$, also belongs to  $L^2(\Omega)$. Then, for any open bounded set $\Omega$, any element $u \in W_\theta(\Omega)$ has a directional trace on $\partial\Omega$, denoted in this paper $\gamma_\theta u$. The existence of a boundary operator follows from the fact that the directional derivative is a skew-symmetric operator (see Remark \ref{rem:skew}). The fact that it can be formulated as functions defined a.e. for the measure $\mu_\theta$, called the directional measure in the direction $\theta$ on the boundary (see Definition \ref{def:dirtrace}), is the extension to $W_\theta(\Omega)$ of results proved in \cite[Theorems 1 and 2]{mazya}, or in \cite[Theorem 2 p. 164]{evans2015meas}: any element of $W_\theta(\Omega)$ restricted to lines parallel to $\theta$ defines almost everywhere an element of $H^1(\omega)$ for some open set $\omega \subset \mathbb{R}$.

\medskip

Then the directional trace $\gamma_\theta u$ of $u$ in the direction $\theta$ is an element of the space $L^2(\partial\Omega,\mu_\theta)$, where the measure $\mu_\theta$ is supported by the subset of $\partial\Omega$ containing the points which are accessible from $\Omega$ in the direction  $\theta$ (see Definition \ref{def:dirtrace} and Appendix \ref{sec:propdirectionalmeasures} for the properties of these measures, in particular the link with the $(d-1)$-dimensional Lebesgue measure, denoted here by $\lambda^{d-1}$). We then remark that this support is much smaller than $\partial\Omega$ in cases where this boundary is irregular (think about 1D domains whose boundary is a Cantor set). It is worth noting that the following integration-by-parts formula always holds for all elements  $u,v\in W_\theta(\Omega)$:

\begin{equation}\label{eq:intpartintro}
    \int_\Omega \Big( u(x) \partial_\theta v(x) + v(x) \partial_\theta u(x) \Big) {\; \rm d}x = \int_{\partial\Omega} \frac{\gamma_{\theta}u(z) \gamma_{\theta}v(z) - \gamma_{\minus\theta}u(\Sym(z))\gamma_{\minus\theta} v(\Sym(z))}{|z- \Sym(z)|} {\ \rm d}\mu_\theta(z),
\end{equation}
where $\Sym(z)$ is the first intersection of $\partial\Omega$ with the half-line starting from $z$ in the direction $-\theta$. This integration-by-parts shows a surprising form: one cannot integrate the boundary function $\gamma_\theta u \cdot \gamma_\theta v$ alone. The only integrable expression combines the values of the directional trace on two opposite points of chords parallel to $\theta$. The example~\ref{exa:cuspidal} of cuspidal domains provides an explanation for this fact: in severe cusps, the function $u$ may tend to infinity sharply  (such questions are discussed in \cite{poborchi}), but the fact that $\partial_\theta u$ also belongs to $L^2(\Omega)$ implies that the combination of values at opposite points provides integrable expressions. 

\medskip

Nevertheless, our goal is to define a trace operator for the elements of  $H^1(\Omega)$, which are elements of $W_\theta(\Omega)$ for all direction  $\theta \in \mathbb{S}^{d-1}$. It is then natural to introduce in Section~\ref{sec:tracehun} the set $H^1_{\mathrm{tr}}(\Omega)$ of all $u \in H^1(\Omega)$ for which there exists a class of functions, denoted $\mathrm{tr}(u)$, coinciding $\mu_\theta$-a.e.\ with $\gamma_\theta u$ for all $\theta \in \mathbb{S}^{d-1}$. For this reason, we call it the omnidirectional trace. As noticed above, in the case of domains with cracks, if $u$ is discontinuous across the crack, then the directional traces on the crack do not coincide, and such functions do not belong to $H_{\rm tr}^1(\Omega)$, defined as the domain of the trace operator ${\rm tr}$. 
A key tool in this work is to use the fact that an $H^1(\Omega)$ function is a 1D $H^1$ function on almost all lines which are parallel to a given direction. The continuous representative of such a function enables the definition of its trace (in this direction) on the boundary of $\Omega$.
In this spirit, we don't obtain any additional information from the quasi-continuity property of the elements of $H^1(\Omega)$ in the sense of \cite{evans2015meas,mazya}: the Lebesgue points at the boundary which are naturally involved in our framework are the directional ones, available almost everywhere and not quasi-everywhere (see Theorems \ref{thm:lebpoints} and \ref{thm:omnilebpoints}).

\medskip

For $u\in H^1_{\mathrm{tr}}(\Omega)$, we show that ${\rm tr}(u)$ is an element of the space $L^2(\partial\Omega,(\mu_\theta)_{\theta\in\mathbb{S}^{d-1}})$, defined as the set of all equivalence classes of functions belonging to all the $\mathcal{L}^2(\partial\Omega,\mu_\theta)$, whose norm in this space is bounded independently of $\theta$ (and not almost all the $\mathcal{L}^2(\partial\Omega,\mu_\theta)$,  see Definition \ref{def:tracemono}). Using the fact that the space $L^2(\partial\Omega,(\mu_\theta)_{\theta\in\mathbb{S}^{d-1}})$ is a Banach space (see Appendix \ref{sec:intfam}), we prove Theorem \ref{thm:huntclosed} which states that $H_{\rm tr}^1(\Omega)$ is a Hilbert space. It can then be used in Section \ref{sec:proptrace} to give the space $H^{1/2}(\partial\Omega)$ the structure of a Hilbert space. Variational problems in general domains are then naturally formulated in $H_{\rm tr}^1(\Omega)$, since the trace of $u \in H_{\rm tr}^1(\Omega)$, equal to the directional traces by construction, satisfies the integration by parts formula \eqref{eq:intpartintro} for all direction $\theta$.

\medskip

A major consequence of the closedness of $H^1_{\mathrm{tr}}(\Omega)$ in $H^1(\Omega)$ is that it contains $\widetilde{H}^1(\Omega)$, the closure of $H^1(\Omega) \cap C^0(\overline{\Omega})$ in $H^1(\Omega)$. This shows that the limit, in $L^2(\partial\Omega, (\mu_\theta)_{\theta \in \mathbb{S}^{d-1}})$, of the restriction to $\partial\Omega$ of functions in $H^1(\Omega) \cap C^0(\overline{\Omega})$ always exists, is always unique, and can always be used in an integration-by-parts formula and in variational problems. We also prove, in the one-dimensional case (see Section~\ref{sec:oned}), that $H^1_{\mathrm{tr}}(\Omega) = \widetilde{H}^1(\Omega)$ whenever $\Omega \subset \mathbb{R}$. However, we construct an example of a two-dimensional domain $\Omega$ and an element of $H^1_{\mathrm{tr}}(\Omega)$ that does not belong to $\widetilde{H}^1(\Omega)$ (see Lemma~\ref{lem:bicone}).

\medskip

A second consequence is that, when $\widetilde{H}^1(\Omega) = H^1(\Omega)$, then $H_{\rm tr}^1(\Omega) = H^1(\Omega)$. This occurs in the following situations: in the case where $\Omega$ is a Lipschitz domain, which implies that the trace constructed in this paper coincides with the standard trace; in the more general case where $\Omega$ has a continuous boundary (in the sense that continuous local charts exist) \cite[Theorem 2 p.11]{mazya}; or in the case where $\Omega$ is bounded by a Jordan curve \cite[Theorem 1 p. 256]{lewis}. This includes, for example, the Koch snowflake domain \cite{kaz2024koch} and many other domains with fractal boundaries — cases where $\mathcal{H}^{d-1}$ is locally infinite.

\medskip

{\it To avoid risks of confusion with any ordered pair of values, we denote in this paper the open real intervals by $]a,b[$ rather than $(a,b)$.}

\section{Directional traces}\label{sec:traceund}

\subsection{Definition and existence of the directional traces}\label{sec:defexisttraceund}
 
For a given $\theta\in \mathbb{S}^{d-1}$, the directional derivative along $\theta$ of any function $\varphi\in C^\infty_c(\Omega)$, denoted in this paper by $\partial_\theta \varphi$, is defined by
\begin{equation}\label{eq:defdirder} \forall x \in \Omega, \quad \partial_\theta \varphi(x) 
	= \lim_{h\to 0} \frac {\varphi(x+h\theta)-\varphi(x)} h
	= \theta \cdot \nabla \varphi(x), \end{equation}
using the Euclidean scalar product in $\mathbb{R}^d$. For all $u\in L^2(\Omega)$, we say that $u$ admits a \textit{directional derivative along $\theta$ in the weak sense} if there exists $v\in L^2(\Omega)$ such that
\begin{equation}\label{eq:defadj}
	\forall \varphi \in C^\infty_c(\Omega), \quad
	\int_\Omega u(x) \partial_\theta \varphi(x) {\; \rm d}x 
		= - \int_\Omega v(x) \varphi(x) {\; \rm d}x.
\end{equation}
In this case, $v$ is uniquely defined, and we denote by $\partial_\theta   u = v$. Note that, in the case of $u \in  C^\infty_c(\Omega)$, we retrieve the same function as the one defined by \eqref{eq:defdirder}. We then define
\[ W_\theta(\Omega) 
	= \{ u \in L^2(\Omega) : \partial_\theta u \in L^2(\Omega) \}. \]
Then $W_\theta(\Omega)$ is a Hilbert space, with the scalar product
\[ \langle u,v \rangle_\theta = \int_\Omega \big(u(x) v(x) + \partial_\theta u(x) \partial_\theta v(x)\big) {\; \rm d}x. \]
For all $\theta \in \mathbb{S}^{d-1}$, we denote by $\mathcal{P}_\theta : x \mapsto x - (x\cdot\theta) \theta$ the orthogonal projection on the hyperplane $H_\theta = \{y \in \mathbb R^d : y \cdot \theta = 0\}$ orthogonal to $\theta$ passing through $0$. Moreover, for all $x \in \Omega$, let
\[ \delta_\theta(x) = \sup\{s \ge 0 : \forall t \in [0,s[, \ x+t\theta \in \Omega\}, \quad \Phi_\theta(x) = x + \delta_\theta(x) \theta \in \partial\Omega, \quad  \ell_\theta(x) = \delta_\theta(x) + \delta_{\minus\theta}(x). \]

Let $\partial_\theta\Omega =  \Phi_\theta(\Omega) \subset \partial\Omega$ be the part of the boundary that the interior of $\Omega$ sees in the direction $\theta$. Note that the functions $\Phi_\theta$, $\Phi_{\minus\theta}$ and $\ell_\theta$ are each constant on the open segment $\{x+s\theta, -\delta_{\minus\theta}(x)<s<\delta_\theta(x)\}$. If we denote by $z \in \partial_\theta\Omega$ the point $\Phi_\theta(x)$ and by $\widehat z \in \partial_{\minus\theta}\Omega$ the point $\Phi_{\minus\theta}(x)$, then we can uniquely extend by continuity the functions $\Phi_{\minus\theta}$ and $\ell_{\minus\theta}$ to the point $z$, by defining
\begin{equation}\label{eq:defphi}
    \Sym : \begin{cases} \partial_\theta \Omega & \to \partial_{\minus\theta} \Omega \\ \;\; z & \mapsto \;\; \widehat z, \end{cases} \quad
    \hbox{ and } \quad
    \ell_{\minus\theta} : \begin{cases} \partial_\theta \Omega & \to ]0,+\infty[ \\ \;\; z & \mapsto \;\; |\widehat z - z|. \end{cases}
\end{equation}
The map $\Sym$ is bijective (see Lemma \ref{lem:visavis}), and we deduce that $\ell_\theta(\Sym(z)) = \ell_{\minus\theta}(z)$. To finish, let us define
\[ \omega_\theta(y) = \{s \in \mathbb{R} : s\theta+y \in \Omega\} \subset \mathbb{R}, \quad y \in H_\theta. \]
Note that, for all $y \in \mathcal{P}_\theta(\Omega)$, $\omega_\theta(y)$ is a non-empty open subset of $\mathbb{R}$. We then denote by  $\mathcal{I}_{\theta}(y)$ the non-empty countable set of disjoint bounded open intervals such that $\omega_\theta(y) = \bigcup_{]\alpha,\beta[\in \mathcal{I}_{\theta}(y)} ]\alpha,\beta[$. Some properties of these objects are detailed in Appendix \ref{sec:propdirectionalmeasures}. \medskip

Let us observe that $\Omega = \{y + s\theta : y \in H_\theta, s \in \omega_\theta(y)\}$, and
\[ W_\theta(\Omega) = W_{\minus\theta}(\Omega), \quad
	\partial_{\minus\theta} = -\partial_\theta, \quad
	H_\theta = H_{\minus\theta}, \quad
	\mathcal{P}_\theta = \mathcal{P}_{\minus\theta}, \quad
	\langle u,v\rangle_\theta =\langle u,v\rangle_{\minus\theta}. \]
Also, note that $]\alpha,\beta[ \ \in \mathcal{I}_{\theta}(y)$ if and only if $]-\beta,-\alpha[ \ \in \mathcal{I}_{\minus\theta}(y)$. Finally, notice that $\partial_\theta\Omega$ and $\partial_{\minus\theta}\Omega$ are different in the general case, even though it is possible to have $\partial_\theta\Omega =  \partial_{\minus\theta}\Omega$, for instance in 1D (with $\theta \in \mathbb{S}^0 = \{-1,1\}$), $\Omega = (\bigcup_{n\in\mathbb{N}^\star} ]-1+\frac 1 {n+1}, -1+\frac 1 n[) \cup (\bigcup_{n\in\mathbb{N}^\star} ]1-\frac 1 n, 1-\frac 1 {n+1}[) \subset \mathbb R$.
 
\begin{definition}[Directional trace in the direction $\theta$]\label{def:dirtrace}
	Let $u\in W_\theta(\Omega)$. We say that a function $g : \partial\Omega \to \mathbb{R}$ is a directional trace of $u$ in the direction $\theta$ if there exists a representative of $u$, again denoted by $u$, and there exists a set $A \subset \mathcal{P}_\theta(\Omega)$ such that
	\begin{equation} \label{eq:deftradir}\left. \begin{aligned}
		1. \quad & \mathcal{P}_\theta(\Omega) \setminus A \text{ is } \lambda^{d-1}\text{-negligible}, \\
		2. \quad & \forall y \in A, \  u_{\theta,y} : s \ \mapsto u(s\theta+y) \text{ is an element of } H^1(\omega_\theta(y)), \\
		&\phantom{\forall y \in A, } \hbox{and } \ \forall \ ]\alpha,\beta[ \ \in \mathcal{I}_{\theta}(y), \ u_{\theta,y}\in C^0(]\alpha,\beta])\hbox{ and }g(\beta\theta+y) = u_{\theta,y}(\beta).
	\end{aligned} \right\} 
	\end{equation}
	 We then define the directional trace of $u$ in the direction $\theta$, denoted by $\gamma_\theta u$, as the set of all directional traces of $u$ in the direction $\theta$ in the preceding sense.
\end{definition}
\begin{remark} Using \eqref{eq:valbeta} in Lemma \ref{lem:fhun}, the left-continuity of $u_{\theta,y}$ in $\beta$ leads to:
\[
u_{\theta,y}(\beta) =  \frac 1 {\beta - \alpha}\int_\alpha^\beta (u_{\theta,y}(t) + (t-\alpha) u_{\theta,y}'(t)) {\rm d}t.
\]
If there exist reals $\alpha_1<\alpha_2<\alpha_3$ such that $]\alpha_1,\alpha_2[ \; \in \mathcal{I}_{\theta}(y)$ and  $]\alpha_2,\alpha_3[ \; \in \mathcal{I}_{\theta}(y)$, the function $u_{\theta,y}$ may be discontinuous in $\alpha_2$.
\end{remark}

The following theorem guarantees that the set of directional traces of $u$ in the direction $\theta$ is not empty.

\begin{theorem}[Existence of a directional trace]
    Let $u \in W_\theta(\Omega)$. There exists a directional trace $g :\partial\Omega \to \mathbb{R}$ of $u$ in the direction $\theta$ in the sense of Definition \ref{def:dirtrace}, which moreover is measurable for the Borel $\sigma$-algebra.
\end{theorem}

This theorem as well as the subsequent results remain valid upon replacing $\theta$ by $-\theta$.

\begin{proof}
 Up to a change of coordinates and the use of an isometry from $\mathbb R^{d-1}$ to $H_\theta$, we treat the case $\theta = e_1 = (1, 0, \cdots, 0)$. Let $u \in W_{e_1}(\Omega)$. 
 
 \medskip

 We use the same notation for representatives of $u$ and $\partial_{e_1}  u$ such that $u\in {\mathcal L}^2(\Omega )$ and  $\partial_{e_1}  u\in {\mathcal L}^2(\Omega )$. For proving the theorem, let us show that, for a.e. $y \in \mathcal{P}_{e_1}(\Omega)$, the functions $u_{e_1,y}:=u(\cdot,y)$ and $\partial_{e_1} u(\cdot,y) $ are such that
\begin{align}
& u_{e_1,y} \in L^2(\omega_1(y)),\label{eq:itemun}
\\
& \partial_{e_1} u(\cdot,y)\in L^2(\omega_1(y)),\label{eq:itemdeux}
\\
&u_{e_1,y}' = \partial_{e_1} u(\cdot,y).\label{eq:itemtrois}
\end{align}

\medskip

Extending $u$ by $0$ outside $\Omega $ and applying Fubini-Tonelli's theorem, we can write
\[
\int_{{\mathbb R}^{d-1}} \Big( \int_{\mathbb R} u(x,y)^2 {\rm d}x \Big) {\rm d}y =\int_{\Omega } u(z)^2 {\; \rm d}z < +\infty.
\]
A more precise consequence of Fubini-Tonelli's theorem is that the mapping $y \mapsto \int_{\mathbb R} u(x,y)^2 {\; \rm d}x$ is a measurable (in the sense that it is a.e. equal to a measurable function) and integrable mapping from ${\mathbb R}^{d-1}$ to ${\mathbb R}$. We therefore have, for a.e.  $y\in\mathcal{P}_{e_1}(\Omega)$, $u(\cdot,y) \in L^2(\omega_{e_1}(y))$. This concludes the proof of \eqref{eq:itemun}.

The same reasoning, applied to  $\partial_{e_1} u$, provides the proof of \eqref{eq:itemdeux}.

\medskip

Define $\mathcal{U}$ by
 \[ \mathcal{U} = \left\{\prod_{k=1}^d \ ]\alpha_k, \beta_k[ \ \subset \Omega : \alpha_k< \beta_k \in \mathbb Q \right\}. \]
 Then $\mathcal{U}$ is a countable (since $\mathbb Q$ is countable) cover of $\Omega$. Let $K \in \mathcal{U}$. We can write
 \[ K = \ ]\alpha_1, \beta_1[ \ \times L, \quad
     L = \prod_{k=2}^d \ ]\alpha_k, \beta_k[. \]
 Let $A = \{y \in L : u(\cdot, y), \ \partial_{e_1}u(\cdot, y) \in L^2(\omega_{e_1}(y))\}$ and $\rho \in C^\infty_c(\mathbb R)$ such that
 \[ \rho \ge 0, \quad \big(|x| > 1 \implies \rho(x) = 0\big), \quad \int_{\mathbb R} \rho(x) \ {\rm d} x = 1. \]
 For $n \in \mathbb{N}$ and $y \in A$, we define, for all $x\in \; ]\alpha_1, \beta_1[$,
 \[ u_{n,y}(x) = (u(\cdot, y) * \rho_n)(x) = \int_{\alpha_1}^{\beta_1} u(s,y)\rho_n(x-s) \ {\rm d} s, \quad \text{ where } \quad \rho_n(x) = n\rho(nx). \]
 Using convolution properties, we have $u_{n,y} \in C^\infty(]\alpha_1, \beta_1[)$, and
 \[ u_{n,y}'(x) = \int_{\alpha_1}^{\beta_1} u(s,y) \rho_n'(x-s) \ {\rm d} s. \]
 If $y \not\in A$, we set $u_{n,y} = 0$. Let $m \in \mathbb{N}$ be such that $m> \frac 2 {\beta_1 - \alpha_1}$, $J_m = \ ]\alpha_1 + \frac 1 m, \beta_1 - \frac 1m[$, $\varphi \in C^\infty_c(L)$, and $\psi(s,y) = \rho_n(x-s)\varphi(y)$. Then for all $x \in J_m$ and $n > 2m$, the support of $\psi$ is a subset of $K$, and
 \begin{align*}
     \int_L u_{n,y}'(x) \varphi(y) \ {\rm d} y
     & = \int_K u(s,y) \rho_n'(x-s) \varphi(y) \ {\rm d} s \ {\rm d} y \\
     & = -\int_K u(s,y) \partial_{e_1}\psi(s,y) \ {\rm d} s \ {\rm d} y
     = \int_K \partial_{e_1} u(s,y) \psi(s,y) \ {\rm d} s \ {\rm d} y,
 \end{align*}
 using the integration-by-parts formula. Repeating this step with
 \[ w(x,y) = \int_{\alpha_1}^{x} \partial_{e_1}u(s,y) \ {\rm d} s \quad
     \text{ and } \quad w_{n,y}(x) = (w(\cdot, y) * \rho_n)(x) = \int_{\alpha_1}^{\beta_1} w(s,y)\rho_n(x-s) \ {\rm d} s, \]
 we obtain (once again using the integration-by-parts formula),
 \[ w_{n,y}'(x) = \int_{\alpha_1}^{\beta_1} w(s,y) \rho_n'(x-s) \ {\rm d} s = \int_{\alpha_1}^{\beta_1} \partial_{e_1}u(s,y) \rho_n(x-s) \ {\rm d} s, \]
 so that
 \begin{align*}
     \int_L w_{n,y}'(x) \varphi(y) \ {\rm d} y
     = \int_K \partial_{e_1}u(s,y) \psi(s,y) \ {\rm d} s \ {\rm d} y
     = \int_L u_{n,y}'(x) \varphi(y) \ {\rm d} y.
 \end{align*}
 Hence, $(w_{n,y} - u_{n,y})'(x) = 0$ for a.e. $y \in L$, so there exists a $\lambda^{d-1}$-negligible set $E_x \subset L$ such that $(w_{n,y} - u_{n,y})'(x) = 0$ for all $y \in L \setminus E_x$. Let $E$ be the union of the $E_x$, with $x \in \mathbb Q \cap J_m$. Then $E$ is negligible (because $\mathbb Q$ is countable, so $E$ is a countable union of negligible sets), and for all $y \not\in E$ and $x \in \mathbb Q \cap J_m$, we have $(w_{n,y} - u_{n,y})'(x) = 0$. However, $\mathbb Q \cap J_m$ is a dense subset of $J_m$, and $w_{n,y} - u_{n,y}$ is a $C^\infty$ map on $]\alpha_1, \beta_1[$. Thus,
 \[ \forall y \in L \setminus E, \; \forall x \in J_m, 
     \; (w_{n,y} - u_{n,y})'(x) = 0. \]
 Finally, for a.e. $y \in L$, $w_{n,y} - u_{n,y}$ is constant on $J_m$. Since $(\rho_n)_{n \ge 0}$ is a mollifier, we obtain the following identity (for the $L^2$ norm):
 \[ \lim_{n \to +\infty} (w_{n,y} - u_{n,y})
     = \lim_{n \to +\infty} [(w(\cdot, y) - u(\cdot, y)) * \rho_n]
     = w(\cdot, y) - u(\cdot, y). \]
 Hence, $w(\cdot, y) - u(\cdot, y)$ is constant on $J_m$. Letting $m \to +\infty$, the map is constant on $]\alpha_1, \beta_1[$, so there exists $c \in \mathbb R$ such that $u(\cdot, y) = w(\cdot, y) + c$. \medskip

 To conclude, for all $K = \ ]\alpha_{1,K}, \beta_{1,K}[ \ \times L_K \in \mathcal{U}$, there exists a $\lambda^{d-1}$-negligible set $F_K \subset L_K$ such that for all $y \in L_K \setminus F_K$, $u(\cdot, y) \in H^1(]\alpha_{1,K}, \beta_{1,K}[)$, and the weak derivative of $u(\cdot, y)$ on $]\alpha_{1,K}, \beta_{1,K}[$ is $\partial_{e_1}u(\cdot, y)$. Since $\mathcal{U}$ is countable, this implies that
 \[ F = \bigcup_{K \in \mathcal{U}} F_K \subset \mathcal P_{e_1}(\Omega) \]
 is also a $\lambda^{d-1}$-negligible set. 

 \medskip
 
 Then, for all $y \in \mathcal P_{e_1}(\Omega) \setminus F$, we have
 \[ \forall K \in \mathcal{U}, \ ]\alpha_{1,K}, \beta_{1,K}[ \ \subset \omega_{e_1}(y) \implies u(\cdot, y) \in H^1(]\alpha_{1,K}, \beta_{1,K}[), \]
 and the weak derivative of $u(\cdot, y)$ on $\omega_{e_1}(y)$ is $\partial_{e_1}u(\cdot, y)$. 
Using  $u(\cdot, y)$ and $\partial_{e_1}u(\cdot, y) \in L^2(\omega_{e_1}(y))$, we get that the function $u_{e_1,y}(s) = u(s, y)$ for a.e. $s \in \ \omega_{e_1}(y)$ is such that $u_{e_1,y}\in H^1(\omega_{e_1}(y))$ with $u_{e_1,y}' = \partial_{e_1}u(\cdot, y)$, which concludes the proof of \eqref{eq:itemtrois}.
 
 \medskip
 
 We select a representative of $u_{e_1,y}$, still denoted by $u_{e_1,y}$, which is continuous on $]\alpha,\beta]$ for all $]\alpha,\beta[ \ \in \mathcal{I}_{e_1}(y)$ and define $g : \partial\Omega \to \mathbb{R}$ by setting, for all $y \in \mathcal P_{e_1}(\Omega) \setminus F$ and all $]\alpha,\beta[ \ \in \mathcal{I}_{e_1}(y)$, 
 \begin{equation}\label{eq:tracebeta}
  g(\beta, y) = u_{e_1,y}(\beta) = \frac 1 {\beta - \alpha}\int_\alpha^\beta \Big( u_{e_1,y}(t) + (t-\alpha) u_{e_1,y}'(t) \Big) {\; \rm d}t,
 \end{equation}
 in which we use \eqref{eq:valbeta} in Lemma \ref{lem:fhun}.
 Prescribing $g(z) = 0$ for all other $z \in \partial\Omega$, $g$ is a directional trace of $u$ along $e_1$. 
 
 \medskip

 We then observe that, for all $s \in \ ]\alpha,\beta[$ and $x = (s,y)$, we have $\Phi_{\minus e_1}(x) = (\alpha, y)$, $\Phi_{e_1}(x) = (\beta, y)$ and $\beta-\alpha = \delta_{\minus e_1}(x)+\delta_{e_1}(x)$. Since \eqref{eq:tracebeta} may also be written as:
 \begin{equation}\label{eq:traceoned}
     g(\Phi_{e_1}(x)) = \frac 1 {\delta_{\minus e_1}(x)+\delta_{e_1}(x)}\int_{\minus\delta_{\minus e_1}(x)}^{\delta_{e_1}(x)} \Big(u(t, y) + (t+\delta_{\minus e_1}(x))\partial_{e_1} u(t, y)\Big){\ \rm d}t,
 \end{equation}
 the measurability of $g$ for the Borel $\sigma$-algebra follows from that of the functions $\delta_{e_1}$ and $\delta_{\minus e_1}$ proved in Lemma \ref{lem:measr}.
\end{proof}

In Definition \ref{def:dirtrace}, one may in fact choose any representative of $u$, as stated in the following lemma.

\begin{lemma}\label{lem:traceoned}
    Let $u\in  W_\theta(\Omega)$ and let $g:\partial\Omega\to\mathbb{R}$ be a directional trace in the direction $\theta$ in the sense of Definition \ref{def:dirtrace}. Let $u_1$ be a representative of $u$ and let $A_1\subset\mathcal{P}_\theta(\Omega)$ be such that \eqref{eq:deftradir} holds with $u=u_1$ and $A=A_1$. Then, for all representative $u_2$ of $u$, there exists $A_2\subset\mathcal{P}_\theta(\Omega)$ such that \eqref{eq:deftradir} holds with $u=u_2$ and $A = A_2$.
\end{lemma}

\begin{proof}
	Let $B \subset \mathcal{P}_\theta(\Omega)$, whose complement in $\mathcal{P}_\theta(\Omega)$ is $\lambda^{d-1}$-negligible, be such that, for all $y \in B$, $u_1(s\theta+y) = u_2(s\theta+y)$ for a.e. $s\in \omega_\theta(y)$. Then for all  $y\in A_2 := B \cap A_1$ (whose complement in  $\mathcal{P}_\theta(\Omega)$ is therefore $\lambda^{d-1}$-negligible), for all $]\alpha,\beta[ \ \in \mathcal{I}_{\theta}(y)$, the function $f : s \in \ ]\alpha,\beta[ \ \mapsto u_1(s\theta+y) \in \mathbb{R}$ is such that $f(s) = u_2(s\theta+y)$ for a.e. $s \in \ ]\alpha,\beta[$ and $f\in H^1(]\alpha,\beta[)$. Also denoting by  $f$ the continuous representative of $f$ on  $]\alpha,\beta]$, we have $g(\beta\theta+y) = f(\beta)$. Hence \eqref{eq:deftradir} holds with $u = u_2$ and $A = A_2$.
\end{proof}

\subsection{Directional traces and directional measures}\label{sec:directracemesure}

Let $\theta \in \mathbb{S}^{d-1}$. We now introduce a measure on $\partial\Omega$, called the \textit{directional measure} in the direction $\theta$, which plays an essential role in the functional properties of the directional trace.

\begin{definition}[Directional measure on the boundary]\label{def:dirmeasurebound}
	We define the measure $\mu_{\theta}$ on $\partial \Omega$ by
	\begin{equation}\label{eq:defmutheta}
		\mu_{\theta}(A) = \int_{\Omega} \chi_A( \Phi_\theta(x)) {\; \rm d}x,
		\text{ for all } A \in {\mathcal B}(\partial \Omega),
	\end{equation}
	where $\chi_A$ is the characteristic function of $A$ and ${\mathcal B}(\partial \Omega)$ is the set of all Borel sets of $\partial \Omega$ (that is the intersection with $\partial \Omega$ of Borel sets of $\mathbb R^d$). Then $\mu_{\theta}$ is a finite Borel measure, since $\mu_{\theta}(\partial\Omega) \le \lambda^d(\Omega) < +\infty$.
\end{definition}

The properties of this measure are detailed in Appendix \ref{sec:propdirectionalmeasures}. The following theorem provides a link between the notions of directional trace and directional measure.

\begin{theorem}[Directional traces and measure $\mu_\theta$]\label{thm:deftraceoned}
    Let $u \in W_\theta(\Omega)$, and let $g$ be a directional trace of $u$ in the direction $\theta$ in the sense of Definition \ref{def:dirtrace}. Any function $\widetilde{g}:\partial\Omega\to\mathbb{R}$ is a directional trace of  $u$ in the direction $\theta$ if and only if $\widetilde{g}(z) = g(z)$ for $\mu_\theta$-a.e. $z \in \partial\Omega$. As a consequence, the directional trace $\gamma_\theta u$ of $u$ in the direction $\theta$ is an equivalence class (for the relation $\mu_\theta$-a.e. equal) of a measurable function for the Borel $\sigma$-algebra. 
\end{theorem}

\begin{proof}
    \uline{Step 1}: We show that all directional traces are $\mu_\theta$-a.e. equal. \medskip
    
    Let $g$ and $\widetilde{g}$ be two directional traces of $u$ in the direction $\theta$ in the sense of Definition \ref{def:dirtrace}. Let $u$ again denote a representative of $u$, and let $A$ (resp. $\widetilde{A}$) be a subset of $\mathcal{P}_\theta(\Omega)$ such that \eqref{eq:deftradir} hold for $g$, $u$ and $A$ on one hand, $\widetilde{g}$, $u$ and $\widetilde{A}$ on the other hand (owing to Lemma \ref{lem:traceoned}, we can consider the same representative of $u$ for satisfying \eqref{eq:deftradir} with $g$ and $\widetilde{g}$). Then the complement of $A \cap \widetilde{A}$ in $\mathcal{P}_\theta(\Omega)$ is $\lambda^{d-1}$-negligible, and, for all $y \in A \cap \widetilde{A}$ and for all $]\alpha,\beta[ \ \in \mathcal{I}_{\theta}(y)$, then $g(\beta\theta+y) = \widetilde{g}(\beta\theta+y)$. Using Lemma \ref{lem:partialthetaboundary} and Corollary \ref{cor:negl2}, this implies that $g(z) = \widetilde{g}(z)$ for $\mu_\theta$-a.e. $z \in \partial\Omega$. \medskip
    
    \uline{Step 2}: We show that any function $\mu_\theta$-a.e. equal to a directional trace is a directional trace.
    
    \medskip
    Let $g$ be a directional trace of $u$ in the direction $\theta$. Let  $\widetilde{g}:\partial\Omega\to\mathbb{R}$ be such that $g(z) = \widetilde{g}(z)$ for $\mu_\theta$-a.e. $z\in\partial\Omega$. We then define
    \[ \widehat{A} = \{y \in \mathcal{P}_\theta(\Omega) : \forall \ ]\alpha,\beta[ \ \in \mathcal{I}_\theta(y), \ g(\beta\theta + y) = \widetilde{g}(\beta\theta + y)\}. \] 
    From Lemma \ref{lem:negl}, we get that the complement of $\widehat{A}$ in $\mathcal{P}_\theta(\Omega)$ is $\lambda^{d-1}$-negligible. 
    Considering a representative of $u$ again denoted by $u$, let $A$ be a subset of $\mathcal{P}_\theta(\Omega)$ such that \eqref{eq:deftradir} is satisfied for $g$, $u$ and $A$. Then $\mathcal{P}_\theta(\Omega) \setminus (A \cap \widehat{A})$ is $\lambda^{d-1}$-negligible, and for all $y\in A\cap \widehat{A}$, we have $u_{\theta,y}\in H^1(\omega_\theta(y))$. Hence, for all $]\alpha,\beta[ \ \in \mathcal{I}_{\theta}(y)$, since $\beta\theta+y \in \partial_\theta\Omega$ and $y\in \widehat{A}$, we have $ \widetilde{g}(\beta\theta+y) = g(\beta\theta+y) = f(\beta)$, where $f(s) = u(s\theta+y)$ for a.e. $s \in \ ]\alpha,\beta[$ and $f\in C^0(]\alpha,\beta])$.
    Therefore \eqref{eq:deftradir} holds with $\widetilde{g}$, $u$ and $A\cap \widehat{A}$, which shows that $\widetilde{g}$ is a directional trace of  $u$ in the direction $\theta$  in the sense of Definition \ref{def:dirtrace}.
\end{proof}

The operators $\gamma_\theta$ defined by Definition \ref{def:dirtrace} satisfy the following properties.

\begin{proposition}\label{prop:tracedirectionnelle}
    Let $\theta \in \mathbb{S}^{d-1}$ and $u \in W_\theta(\Omega)$. Then $\gamma_{\theta} u \in L^2(\partial\Omega, \mu_{\theta})$ and 
   \begin{equation}\label{eq:ineqtracezero}
        \int_{\partial\Omega} (\gamma_{\theta} u(z))^2 {\ \rm d}\mu_\theta(z) \le 2 \max(1,{\rm diam}(\Omega)^2)\Vert u\Vert_\theta^2,
    \end{equation}
    and if $\gamma_{\theta} u(z) = 0$ for $\mu_\theta$-a.e. $z \in \partial\Omega$, then the following Poincar\'e inequality holds:
    \begin{equation}\label{eq:pointheta}
        \Vert u\Vert_{L^2(\Omega)} 
        \le {\rm diam}(\Omega) \Vert \partial_\theta  u\Vert_{L^2(\Omega)}.
    \end{equation}
\end{proposition}

\begin{proof}
    Let $u \in W_\theta(\Omega)$, and let $g$ be a representative of $\gamma_{\theta} u$ which is a directional trace of $u$ in the sense of Definition \ref{def:dirtrace}. We again denote by $u$ a representative of $u$ and by $A \subset \mathcal{P}_\theta(\Omega)$ a set such that \eqref{eq:deftradir} hold. We have, using the notation of Definition \ref{def:dirtrace} and accounting for \eqref{eq:defmutheta},
    \[ \int_{\partial\Omega} g(z)^2 {\ \rm d} \mu_\theta(z)
            = \int_\Omega g(\Phi_\theta(x))^2 {\ \rm d} x
            = \int_{H_\theta} \int_{\omega_\theta(y)} g(\Phi_\theta(y+s\theta))^2 {\ \rm d}s {\ \rm d}y. \]
        However, if $]\alpha, \beta[ \ \subset \omega_\theta(y)$ and $s \in \ ]\alpha, \beta[$, then $\Phi_\theta(y+s\theta) = y+\beta\theta$. Hence,
        \begin{align*}
            \int_{\partial\Omega} g(z)^2 {\ \rm d}\mu_\theta(z)
            & = \int_{H_\theta}  \sum_{]\alpha,\beta[\in \mathcal I_\theta(y)} (\beta-\alpha) g(y+\beta\theta)^2{\rm d}y.
        \end{align*}
        We know, by Definition \ref{def:dirtrace}, that $g(y+\beta\theta) = f(\beta)$, where $f$ is the continuous representative on $[\alpha, \beta]$ of the map $u_{\theta,y} : s \mapsto u(y+s\theta)$, which lies in $H^1(]\alpha, \beta[)$. Applying \eqref{eq:majfhun}, we get
        \begin{align*}
            \int_{\partial\Omega} g(z)^2 {\ \rm d}\mu_\theta(z) 
            & \le \int_{H_\theta} \sum_{]\alpha,\beta[\in  \mathcal{I}_{\theta}(y)} \Vert u_{\theta,y} \Vert_{H^1(]\alpha,\beta[)}^2 2\max(1, (\beta - \alpha)^2) {\rm d}y \\
            & \le 2 \max(1,{\rm diam}(\Omega)^2) \int_{H_\theta} \Vert u_{\theta,y}\Vert_{H^1(\omega_\theta(y))}^2 {\ \rm d}y.
        \end{align*} 
        Noticing that the weak derivative of $u_{y, \theta}$ on $\omega_\theta(y)$ is $\partial_\theta u(\cdot \theta + y)$, we obtain \eqref{eq:ineqtracezero}. Moreover, if $g(\beta\theta+y)=0$ for all $]\alpha, \beta[ \ \in \mathcal I_\theta(y)$, then \eqref{eq:poinfhun} gives \eqref{eq:pointheta}.
\end{proof}

\begin{proposition}\label{prop:tracesvisavis}
    Let $\theta \in \mathbb{S}^{d-1}$ and $u \in W_\theta(\Omega)$. Then using the map $\Sym$ defined in \eqref{eq:defphi}, we have
    \[ \int_{\partial\Omega} \Big(\gamma_{\theta} u(z) + \gamma_{\minus \theta} u(\Sym(z))\Big)^2 {\ \rm d}\mu_\theta(z)\le 4 \max(1,{\rm diam}(\Omega)^2)\Vert u\Vert_\theta^2, \]
    and
    \[ \int_{\partial\Omega} \Big(\frac{\gamma_{\theta} u(z) -\gamma_{\minus \theta} u(\Sym(z))}{\ell_{\minus\theta}(z)}\Big)^2 {\ \rm d}\mu_\theta(z)\le \Vert u\Vert_\theta^2. \]
\end{proposition}

\begin{proof}
    The first formula is a direct consequence of \eqref{eq:ineqtracezero} and Lemma \ref{lem:visavis}. The second one comes from \eqref{eq:majdiffhun} in Lemma \ref{lem:fhun}, after noticing that
    \[ \int_{\partial\Omega} \Big(\frac{\gamma_{\theta} u(z) -\gamma_{\minus \theta} u(\Sym(z))}{\ell_{\minus\theta}(z)}\Big)^2 {\ \rm d}\mu_\theta(z)
        = \int_{H_\theta} \sum_{]\alpha, \beta[ \in \mathcal I_\theta(y)} \frac{[\gamma_\theta u(y+\beta\theta) - \gamma_{\minus\theta} u(y+\alpha\theta)]^2}{\beta - \alpha}  {\rm d}y, \]
    with $\gamma_{\minus\theta} u(y+\alpha\theta) = f(\alpha)$ and $\gamma_\theta u(y+\beta\theta) = f(\beta)$, where $f$ is the continuous representative on $[\alpha, \beta]$ of the map $u_{\theta,y} : s \mapsto u(y+s\theta)$. 
\end{proof}

\begin{theorem}[Integration-by-parts formula]\label{thm:ipp}
    Let $\theta \in \mathbb{S}^{d-1}$ and $u,v \in W_\theta(\Omega)$. We have
    \begin{equation}\label{eq:greenoned}
        \int_\Omega \Big(u(x)\partial_\theta v(x) +v(x)\partial_\theta u(x)\Big){\ \rm d}x = \int_{\partial\Omega} \frac{\gamma_{\theta} u(z)\gamma_{\theta} v(z) -\gamma_{\minus \theta} u(\Sym(z))\gamma_{\minus \theta} v(\Sym(z))}{\ell_{\minus\theta}(z)} {\ \rm d}\mu_\theta(z).
    \end{equation}
\end{theorem}

\begin{proof}
    We first verify that the right-hand side is well defined. Indeed, denoting for short $\ell_z = \ell_{\minus\theta}(z)$, $a_z = \gamma_{\theta} u(z)$, $b_z = \gamma_{\theta} v(z)$, $c_z = \gamma_{\minus \theta} u(\Sym(z))$ and $d_z = \gamma_{\minus \theta} v(\Sym(z))$, we write 
    \begin{multline*}
        \Big|\frac {a_zb_z-c_zd_z}{\ell_z}\Big| =\Big|  \frac {a_z-c_z}{\ell_z}\frac {b_z+d_z}{2} +  \frac {a_z+c_z}{2}\frac {b_z-d_z}{\ell_z}\Big|\\
        \le  \frac 1 2\Big(\frac {a_z-c_z}{\ell_z}\Big)^2 + \frac 1 2 \Big(\frac {b_z+d_z}{2}\Big)^2 + \frac 1 2 \Big(\frac {a_z+c_z}{2}\Big)^2 +\frac 1 2\Big(\frac {b_z-d_z}{\ell_z}\Big)^2.
    \end{multline*}
    All four terms obtained are $\mu_\theta$-integrable on $\partial \Omega$, by Proposition \ref{prop:tracesvisavis}. We now move on to the calculations. First, see that, for a.e. $y\in \mathcal{P}_\theta(\Omega)$, the functions $s\mapsto u(y+s\theta)$ and $s\mapsto v(y+s\theta)$
    are elements of $H^1(\omega_\theta(y))$. When restricted to $]\alpha, \beta[ \ \in \mathcal I_\theta(y)$, they have continuous representatives on $[\alpha,\beta]$, denoted by $f$ and $g$, which satisfy
    \[
    f(\beta)g(\beta) - f(\alpha)g(\alpha) = \int_\alpha^\beta \big(f(s) Dg(s) + Df(s) g(s)\big) {\rm d}s.
    \]
    Accounting for Definition \ref{def:dirtrace} of the directional trace, we deduce
    \[
    \gamma_{\theta} u(y+\beta\theta)\gamma_{\theta} v(y+\beta\theta) -\gamma_{\minus \theta} u(y+\alpha\theta)\gamma_{\minus \theta} v(y+\alpha\theta) =  \int_\alpha^\beta \big( u(y+s\theta) \partial_\theta v(y+s\theta) + \partial_\theta u(y+s\theta) v(y+s\theta)\big){\ \rm d} s.
    \]
    Summing on $]\alpha, \beta[ \ \in \mathcal I_\theta(y)$ and integrating with respect to $y$, we obtain
    \begin{align*}
        & \int_{H_\theta} \sum_{]\alpha,\beta[\in \mathcal I_\theta(y)} \big( \gamma_{\theta} u(y+\beta\theta)\gamma_{\theta} v(y+\beta\theta) -\gamma_{\minus \theta} u(y+\alpha\theta)\gamma_{\minus \theta}v(y+\alpha\theta)\big){\ \rm d}y \\ = & \int_{H_\theta} \sum_{]\alpha,\beta[\in \mathcal I_\theta(y)} \int_\alpha^\beta \big( u(y+s\theta) \partial_\theta v(y+s\theta) + \partial_\theta u(y+s\theta) v(y+s\theta)\big){\ \rm d} s {\ \rm d}y. 
    \end{align*}
   From the definition of $\mu_\theta$, we get that 
     \begin{align*}
        & \int_{H_\theta} \sum_{]\alpha,\beta[\in \mathcal I_\theta(y)} \big( \gamma_{\theta} u(y+\beta\theta)\gamma_{\theta} v(y+\beta\theta) -\gamma_{\minus \theta} u(y+\alpha\theta)\gamma_{\minus \theta}v(y+\alpha\theta)\big){\ \rm d}y \\ = & \int_{\partial\Omega} \frac{\gamma_{\theta} u(z)\gamma_{\theta} v(z) -\gamma_{\minus \theta} u(\Sym(z))\gamma_{\minus \theta} v(\Sym(z))}{\ell_{\minus\theta}(z)} {\ \rm d}\mu_\theta(z),
    \end{align*}
    and by the change of variable $(s,y)\to x$, we obtain
  \begin{align*}
        &  \int_{H_\theta} \sum_{]\alpha,\beta[\in \mathcal I_\theta(y)} \int_\alpha^\beta \big( u(y+s\theta) \partial_\theta v(y+s\theta) + \partial_\theta u(y+s\theta) v(y+s\theta)\big){\ \rm d} s {\ \rm d}y\\ = &
        \int_\Omega \big(u(x) \partial_\theta v(x) + \partial_\theta u(x) v(x)\big) {\ \rm d} x. 
    \end{align*}
    This concludes the proof of \eqref{eq:greenoned}.
   \end{proof}
 
\begin{remark}
    As observed in the introduction,  $\int_{\partial\Omega} \gamma_{\theta} u\gamma_{\theta}v/\ell_{\minus\theta}\ {\rm d}\mu_\theta $ and $\int_{\partial\Omega} \gamma_{\minus\theta} u \gamma_{\minus\theta} v /\ell_{\theta}\ {\rm d}\mu_{\minus\theta}$
    may separately be non-integrable in \eqref{eq:greenoned}, we need to combine both terms to get integrable functions. Consider the domain in Example \ref{exa:cuspidal}, and the functions $u(x_1,x_2) = v(x_1,x_2) = x_2^{-\alpha}$ with $\alpha\in \ ]\frac 1 2,1[$, then $\frac{\gamma_{\theta} u\gamma_{\theta} v}{\ell_{\minus\theta}} $ is not integrable for the measure $\mu_\theta$.
\end{remark}

\begin{remark}\label{rem:skew}
    Note that $W_\theta(\Omega)$ is the domain of the maximal extension, in $L^2(\Omega)$, of the skew-symmetric operator $-\partial_\theta$ initially defined on $C^\infty_c(\Omega)$. Hence we know from \cite{ace2023ext} that there exist Hilbert spaces $H_\plus$ and $H_\minus$ and two continuous operators $G_\pm:W_\theta(\Omega)\to H_\pm$ such that the mapping $u\mapsto (G_\plus(u),G_\minus(u))$ defined from $W_\theta(\Omega)$ to $H_\plus\times H_\minus$ is surjective, and
    \[ \forall u,v\in W_\theta(\Omega),\ 
    \int_\Omega \big( u(x) \partial_\theta v(x) + v(x) \partial_\theta   u(x)\big){\rm d}x = \langle G_\plus u,G_\plus v\rangle_{H_\plus} - \langle G_\minus u,G_\minus v\rangle_{H_\minus}. \]
    We get from Theorem \ref{thm:ipp} that we can define $H_\pm = G_\pm(W_\theta(\Omega)) = L^2(\partial\Omega,\mu_\theta)$ with $G_\pm:W_\theta(\Omega)\to L^2(\partial\Omega,\mu_\theta)$ given by
    \[ G_\plus u(z) = \frac 1 2 \big(a_z+ b_z + \frac {a_z-b_z} {\ell_z}\big)
        \hbox{ and }
        G_\minus u(z) = \frac 1 2 \big(a_z+ b_z - \frac {a_z-b_z} {\ell_z}\big), \]
    for $\mu_\theta$-a.e. $z\in\partial\Omega$, letting $\ell_z = \ell_{\minus\theta}(z)$, $a_z =  \gamma_{\theta} u(z)$ and $b_z = \gamma_{\minus \theta} u(\Sym(z))$.
    The surjectivity of the mapping $u\mapsto (G_\plus(u),G_\minus(u))$, from $W_\theta(\Omega)$ to $L^2(\partial\Omega,\mu_\theta)^2$, is a consequence of the fact that, for all $a,b\in\mathbb{R}$, the function $f:s\mapsto a + (b-a)\frac {s - \alpha}{\beta - \alpha}$ is such that $f\in H^1(]\alpha,\beta[)\cap C^0([\alpha,\beta])$ with $f(\beta) = b$ and $f(\alpha) = a$ .
\end{remark}

\begin{lemma}[Restriction to an open subset of $\Omega$]\label{lem:included}
    Let $\widehat{\Omega}\subset\Omega$ be an open set.
    We denote by $\widehat{\mu}_\theta$, $\widehat{\Phi}_\theta$, $\widehat{\gamma}_\theta$ the quantities $\mu_\theta$, $\Phi_\theta$, $\gamma_\theta$ replacing $\Omega$ by $\widehat{\Omega}$. Let $u\in H^1(\Omega)$, and let $\widehat{u}$ denote the restriction of $u$ to $\widehat{\Omega}$. Let $g$ (resp. $\widehat{g}$) be a directional trace of $u$ (resp. $\widehat{u}$) in the direction $\theta$ in the sense of Definition \ref{def:dirtrace}. Then $\widehat{g}(z) = g(z)$ for $\mu_\theta$-a.e. and $\widehat{\mu}_\theta$-a.e. $z \in \partial_\theta \widehat{\Omega} \cap \partial_\theta\Omega$.
\end{lemma}

\begin{proof}
    We select a representative of $u$, again denoted by $u$, and we again denote by $\widehat{u}$ the restriction of the representative $u$ to  $\widehat{\Omega}$. Let $A$ (resp $\widehat{A}$) such that \eqref{eq:deftradir} hold.
    Then for all $y \in A$, we have $u(\cdot\theta+y) \in H^1(\omega_\theta(y))$ and for all $y \in \widehat{A}$, we also have $\widehat{u}(\cdot\theta+y) \in H^1(\widehat{\omega}_\theta(y))$. Moreover, for all $y \in A \cap \widehat A$, $u(\cdot\theta+y) = \widehat{u}(\cdot\theta+y)$, and the sets $\mathcal{P}_\theta(\partial_\theta\widehat{\Omega})\setminus \widehat{A}$ and $\mathcal{P}_\theta(\partial_\theta\Omega)\setminus A$ are $\lambda^{d-1}$-negligible. Let us define
    \[ B = \{z\in \partial_\theta\widehat{\Omega} \cap \partial_\theta\Omega : g(z) = \widehat{g}(z)\}, \]
    and
    \[ \widetilde{B} = \{z\in \partial_\theta\widehat{\Omega}\cap \partial_\theta\Omega : \mathcal{P}_\theta(z) \in A\cap \widehat{A}\}. \]
    We get from \eqref{eq:deftradir} that, for $z \in \widetilde{B}$, it holds $g(z) = \widehat{g}(z)$. This proves that $\widetilde{B}\subset B$. \medskip 

    Let $y \in \mathcal{P}_\theta((\partial_\theta\widehat{\Omega}\cap \partial_\theta\Omega) \setminus B)$. We therefore have that, for $z \in (\partial_\theta\widehat{\Omega} \cap \partial_\theta\Omega) \setminus B$ such that $y = \mathcal{P}_\theta(z)$, $g(z) \neq \widehat{g}(z)$, which means that $y\notin A \cap \widehat{A}$. If $y \in A$ (resp. $\widehat A$), this implies that $y \notin \widehat{A}$ (resp. $y \notin A$), which means that $y \in \partial_\theta \widehat{\Omega} \setminus \widehat{A}$ (resp. $y \in \partial_\theta \Omega \setminus A$). If $y \notin A \cup \widehat{A}$, it means that $y \in (\partial_\theta \widehat{\Omega} \setminus \widehat{A}) \cup (\partial_\theta\Omega \setminus A)$. Therefore
    \[ \mathcal{P}_\theta((\partial_\theta\widehat{\Omega} \cap \partial_\theta\Omega) \setminus B) \subset (\mathcal{P}_\theta(\partial_\theta\Omega)\setminus A) \cup (\mathcal{P}_\theta(\partial_\theta\widehat{\Omega})
    \setminus\widehat{A}). \]
    This implies that the set $\mathcal{P}_\theta((\partial_\theta\widehat{\Omega}\cap \partial_\theta\Omega) \setminus B$) is $\lambda^{d-1}$-negligible.
    Therefore, from Corollary \ref{cor:negl2}, we get that
    $(\partial_\theta\widehat{\Omega}\cap \partial_\theta\Omega) \setminus B$ is both $\mu_\theta$-negligible and $\widehat{\mu}_\theta$-negligible.
\end{proof}

\begin{remark}
    Lemma \ref{lem:included} cannot be strengthened to assert $\widehat{g}(z) = g(z)$ for $\mu_\theta$-a.e. and $\widehat{\mu}_\theta$-a.e. $z\in \partial\widehat{\Omega}\cap \partial\Omega$; the conclusion is only valid on $\partial_\theta\widehat{\Omega}\cap \partial_\theta\Omega$.
    Indeed, let us consider the case 
    \[ \Omega = \ ]0,1[, \qquad \widehat{\Omega} = \bigcup_{n\in\mathbb{N}^\star} \left]\frac 1 {2n+1},\frac 1{2n}\right[. \]
    Then, for $\theta = -1$, $\partial_\theta\widehat{\Omega} = \{ \frac 1 {2n+1}, n\in\mathbb{N}^\star\}$, $0\in\partial\widehat{\Omega}$ and  $\partial_\theta\Omega = \{0\}$.
    Then one can take any value for $\widehat{g}(0)$ such that $\widehat{g}(0)\neq g(0)$ (since $\widehat{\mu}_\theta(\{0\}) = 0$) whereas $\mu_\theta(\{0\}) = 1$.
\end{remark}

\subsection{Directional traces and directional Lebesgue points}\label{sec:lebesgue}

In this section, we establish a relation between the directional trace and the average value around some points of the boundary in the direction $\theta\in \mathbb{S}^{d-1}$.

\begin{theorem} [Directional trace and directional Lebesgue points]\label{thm:lebpoints}
Let $\theta \in \mathbb{S}^{d-1}$ and $u\in W_\theta(\Omega)$. For all $\varepsilon>0$ and for $\mu_\theta$-a.e. $z\in \partial_{\theta}\Omega$, we define
\[
u_{\theta,\varepsilon}(z) =  \frac 1 {\min(\varepsilon,\ell_{\minus\theta}(z))} \int_0^{\min(\varepsilon,\ell_{\minus\theta}(z))} u(z - s\theta) {\ \rm d}s.
\]
Then we have
\begin{equation}\label{eq:cvppbord}
 \hbox{For $\mu_\theta$-a.e. }z\in \partial_{\theta}\Omega,\ \lim_{\varepsilon\to 0} u_{\theta,\varepsilon}(z) = \gamma_\theta u(z)
\end{equation}
and
\begin{equation}\label{eq:intbord}
 \int_{\partial_{\theta}\Omega} ( \gamma_\theta u(z) - u_{\theta,\varepsilon}(z))^2 {\ \rm d}\mu_\theta(z) \le \varepsilon\ {\rm diam}(\Omega)\Vert\partial_\theta u\Vert_{L^2(\Omega)}^2.   
\end{equation}
 \end{theorem}
 \begin{proof}
 We select a representative of $u$. Then, for $z\in \partial_{\theta}\Omega$ such that the function $s \mapsto u(z - s \theta)$ belongs to $H^1(\omega_\theta(\mathcal{P}_\theta(z)))$, we have, for all $t \in \ ]0, \min(\varepsilon,\ell_{\minus\theta}(z))[$,
\[
\gamma_\theta u(z) - u(z - t \theta) =\int_0^t \partial_\theta u (z - s\theta){\rm d}s.
\]
Using the Fubini theorem, and denoting for short $\zeta = \min(\varepsilon,\ell_{\minus\theta}(z))$, we get
\[
\gamma_\theta u(z) - u_{\theta,\varepsilon}(z)=\frac 1 \zeta \int_0^\zeta \left(\int_0^\zeta \partial_\theta u (z-s\theta) 1_{]0,t[} (s) {\ \rm d}s\right) {\rm d}t = \frac 1 \zeta \int_0^\zeta \left(\int_0^\zeta 1_{]s,\zeta[} (t) {\ \rm d}t\right) \partial_\theta u (z - s\theta){\ \rm d}s.
\]
This gives

\[
    |\gamma_\theta u(z) - u_{\theta,\varepsilon}(z)| \le  \int_0^\zeta | \partial_\theta u (z -s\theta) | {\ \rm d}s.
\]
Using the Cauchy-Schwarz inequality, we obtain
\begin{equation}\label{maj-bte-square}
    |\gamma_\theta u(z) - u_{\theta,\varepsilon}(z)|^2 \le  \zeta \int_0^\zeta | \partial_\theta u (z -s\theta) |^2 {\ \rm d}s.
\end{equation}
This implies that
\[
|\gamma_\theta u(z) - u_{\theta,\varepsilon}(z)|^2 \le\varepsilon \Vert\partial_\theta u (z - \cdot\theta)\Vert_{L^2(\omega_\theta(\mathcal{P}_\theta(z)))}^2,
\]
which implies \eqref{eq:cvppbord}.

\medskip

We then integrate \eqref{maj-bte-square} for the measure $\mu_\theta$. This yields, from the definition of $\mu_\theta$,
\[
\int_{\partial_\theta \Omega} |\gamma_\theta u(z) - u_{\theta,\varepsilon}(z)|^2 {\ \rm d} \mu_\theta(z) \le 
\int_\Omega \min(\varepsilon,\ell_{\minus\theta}(x)) \int_0^{\min(\varepsilon,\ell_{\minus\theta}(x))} |\partial_\theta u (\Phi_\theta(x) -s\theta) |^2{\ \rm d}s{\ \rm d}x.
\]
We apply the change of variable $x \to (t,y)$ with $x = y+t\theta$ and $t\in \omega_\theta(y)$. This leads to
\begin{multline*}
 \int_\Omega \min(\varepsilon,\ell_{\minus\theta}(x)) \int_0^{\min(\varepsilon,\ell_{\minus\theta}(x))} |\partial_\theta u (\Phi_\theta(x) -s\theta) |^2{\ \rm d}s{\ \rm d}x \\ =
\int_{\mathcal{P}_\theta(\Omega)} \sum_{]\alpha,\beta[\in \mathcal{I}_\theta(y)} \min(\varepsilon,\beta - \alpha) \int_\alpha^\beta \int_0^{\min(\varepsilon,\beta-\alpha)}|\partial_\theta u (y+\beta\theta -s\theta) |^2{\ \rm d}s{\ \rm d}t{\ \rm d}y.   
\end{multline*}
We then have, up to the change of variable $s \mapsto \beta - s$,
\begin{align*}
\int_\alpha^\beta \int_0^{\min(\varepsilon,\beta-\alpha)}|\partial_\theta u (y+\beta\theta -s\theta) |^2{\ \rm d}s{\ \rm d}t & \le
{\rm diam(\Omega)} \int_{\beta - \min(\varepsilon,\beta-\alpha)}^\beta|\partial_\theta u (y+s\theta) |^2{\ \rm d}s \\
& \le {\rm diam(\Omega)} \int_{\alpha}^\beta|\partial_\theta u (y+s\theta) |^2{\ \rm d}s.
\end{align*}
Hence we obtain, since $\min(\varepsilon,\beta - \alpha)\le \varepsilon$,
\begin{align*}
 & \int_\Omega \min(\varepsilon,\ell_{\minus\theta}(x)) \int_0^{\min(\varepsilon,\ell_{\minus\theta}(x))} |\partial_\theta u (\Phi_\theta(x) -s\theta) |^2{\ \rm d}s{\ \rm d}x \\ \le \ & \varepsilon\ {\rm diam(\Omega)}
\int_{\mathcal{P}_\theta(\Omega)} \sum_{]\alpha,\beta[\in \mathcal{I}_\theta(y)} \int_\alpha^\beta |\partial_\theta u (y+s\theta) |^2{\ \rm d}s{\ \rm d}y\\
= \ & \varepsilon\ {\rm diam(\Omega)}\int_\Omega |\partial_\theta u (x) |^2{\ \rm d}x,
\end{align*}
which concludes the proof.
 \end{proof}
 \begin{remark}\label{rem:hunwtheta}
     We observe that Theorem \ref{thm:lebpoints} applies for all function of $H^1(\Omega)$ and for all $\theta\in\mathbb{S}^{d-1}$, since $H^1(\Omega)\subset W_\theta(\Omega)$.
 \end{remark}

\section{Omnidirectional trace operator on $H^1(\Omega)$}\label{sec:tracehun}

\subsection{Definition of $H_{\rm tr}^1(\Omega)$}

Let us observe that
\[ H^1(\Omega) = \bigcap_{\theta\in{\mathbb{S}^{d-1}}} W_\theta(\Omega), \]
and that, for all $\theta\in{\mathbb{S}^{d-1}}$ and $u\in H^1(\Omega)$, we have $\Vert u\Vert_\theta \le \Vert u\Vert_{H^1(\Omega)}$. 
We then denote by
\begin{equation}\label{eq:defdomegatilde}
  \partial_{\mathbb{S}}\Omega  = \bigcup_{\theta\in{\mathbb{S}^{d-1}}} \partial_\theta \Omega \subset\partial\Omega.
\end{equation}
This leads to the following definition.

\begin{definition}
    Let ${\mathcal{L}}^2(\partial\Omega,(\mu_\theta)_{\theta\in\mathbb{S}^{d-1}})$ denote the set of all functions $f:\partial\Omega\to\mathbb{R}$ such that
    \[ \exists M \ge 0, \ \forall \theta \in \mathbb{S}^{d-1}, \ \exists g_\theta \in \mathcal{L}^2(\partial\Omega,\mu_\theta), \ f = g_\theta \ \mu_\theta\text{-a.e. on } \partial\Omega \text{ and } \int_{\partial\Omega} |g_\theta|^2 {\ \rm d}\mu_\theta\le M. \]

    For all $f,g \in {\mathcal{L}}^2(\partial\Omega,(\mu_\theta)_{\theta\in\mathbb{S}^{d-1}})$, we say that $f$ is equivalent to $g$ if, for all $\theta \in \mathbb{S}^{d-1}$, $f(z) = g(z)$ for $\mu_\theta$-a.e. $z \in \partial\Omega$, and we define $L^2(\partial\Omega,(\mu_\theta)_{\theta\in\mathbb{S}^{d-1}})$ as the set of all equivalence classes of the elements of ${\mathcal{L}}^2(\partial\Omega,(\mu_\theta)_{\theta\in\mathbb{S}^{d-1}})$, normed by
    \[ \forall g\in L^2(\partial\Omega,(\mu_\theta)_{\theta\in\mathbb{S}^{d-1}}),\ \Vert g\Vert_{L^2(\partial\Omega,(\mu_\theta)_{\theta\in\mathbb{S}^{d-1}})}^2 = \sup_{\theta\in\mathbb{S}^{d-1}}\int_{\partial\Omega} |g|^2 {\ \rm d}\mu_\theta. \]
\end{definition}

Note that, for $f,g \in {\mathcal{L}}^2(\partial\Omega,(\mu_\theta)_{\theta\in\mathbb{S}^{d-1}})$,  $f$ is equivalent to $g$ if, for all $\theta\in \mathbb{S}^{d-1}$, $f = g$ $\mu_\theta$-a.e. in $\partial_{\mathbb{S}}\Omega $ (defined by \eqref{eq:defdomegatilde}), which is a dense subset of $\partial\Omega$ (see Lemma \ref{lem:closomegatilde}). The set $L^2(\partial\Omega,(\mu_\theta)_{\theta\in\mathbb{S}^{d-1}})$ is a Banach space (this is proved in Appendix \ref{sec:intfam}), but not a Hilbert space. 
Although a measure based on the integration of $\mu_\theta$ with respect to $\theta\in\mathbb{S}^{d-1}$ could lead to the definition of a Hilbert space, it would introduce an almost everywhere sense for $\theta$. In contrast, the framework resulting from the space $L^2(\partial\Omega,(\mu_\theta)_{\theta\in\mathbb{S}^{d-1}})$ leads to integration-by-parts formulas which hold for every $\theta\in\mathbb{S}^{d-1}$, not merely for almost every $\theta$.

\begin{definition}[Omnidirectional trace operator on $H^1(\Omega)$]\label{def:tracemono}
    Let $u \in H^1(\Omega)$. We say that $u$ has an omnidirectional trace on $\partial \Omega$ if there exists a function $g \in L^2(\partial\Omega,(\mu_\theta)_{\theta\in\mathbb{S}^{d-1}})$ such that for all $\theta \in {\mathbb{S}^{d-1}}$, we have $\gamma_\theta u(z) = g(z)$ for $\mu_\theta$-a.e. $z \in \partial\Omega$. In this case, we denote $g = {\rm tr}(u)$, and we define $H_{\rm tr}^1(\Omega)$ as the set of all $u \in H^1(\Omega)$ having an omnidirectional trace on $\partial \Omega$.
\end{definition}

\begin{theorem}[$H_{\rm tr}^1(\Omega)$ is closed]\label{thm:huntclosed}
    The space $H_{\rm tr}^1(\Omega)$ is closed, hence a Hilbert space. Moreover, the operator ${\rm tr} : H_{\rm tr}^1(\Omega) \to L^2(\partial\Omega,(\mu_\theta)_{\theta\in\mathbb{S}^{d-1}})$ is continuous and, for all $u,v \in H_{\rm tr}^1(\Omega)$ and $\theta\in{\mathbb{S}^{d-1}}$, we have:
    \begin{equation}\label{eq:intbypartshun}
        \int_\Omega \Big( u(x) \partial_\theta v(x) +v(x) \partial_\theta u(x)\Big){\ \rm d}x = \int_{\partial\Omega} \frac{{\rm tr }(u)(z) {\rm tr }(v)(z) - {\rm tr }(u)(\Sym(z)){\rm tr }(v)(\Sym(z))}{\ell_{\minus\theta}(z)} {\ \rm d}\mu_\theta(z).
\end{equation}
\end{theorem}

\begin{proof}
    Recall that, using \eqref{eq:ineqtracezero} in Proposition \ref{prop:tracedirectionnelle}, we have, for all $u\in H_{\rm tr}^1(\Omega)$ and any $\theta \in \mathbb{S}^{d-1}$,
    \[ \int_{\partial\Omega} |{\rm tr }(u)|^2 {\ \rm d}\mu_\theta 
        = \int_{\partial\Omega} |\gamma_\theta u|^2 {\ \rm d}\mu_\theta 
        \le 2 \max(1,{\rm diam}(\Omega)^2)\Vert u\Vert_{H^1(\Omega)}^2. \]
    We get that, for every Cauchy sequence $(u_n)_{n\in\mathbb{N}}$ of elements of $H_{\rm tr}^1(\Omega)$ with limit $u\in H^1(\Omega)$, the sequence $({\rm tr }(u_n))_{n\in\mathbb{N}}$ converges in $L^2(\partial\Omega,(\mu_\theta)_{\theta\in\mathbb{S}^{d-1}})$, which is a Banach space by Theorem \ref{theo:lpbanach} in Appendix \ref{sec:intfam}. Denoting by $g$ its limit, we deduce that $g = \gamma_\theta u$ $\mu_\theta$-a.e. on $\partial\Omega$ for all $\theta \in \mathbb{S}^{d-1}$, which shows that $u\in H_{\rm tr}^1(\Omega)$. Equation \eqref{eq:intbypartshun} then follows from Theorem \ref{thm:ipp}.
\end{proof}

We deduce the following corollary.

\begin{corollary}\label{cor:inclusionshuntr}
    Recall that $H^1_0(\Omega)$ denotes the closure of $C^\infty_c(\Omega)$ in $H^1(\Omega)$, and that $\widetilde{H}^1(\Omega)$ denotes the closure of $H^1(\Omega) \cap C^0(\overline{\Omega})$ in $H^1(\Omega)$. We define $H^1_{\rm tr, 0}(\Omega) = \{u \in H_{\rm tr}^1(\Omega) : {\rm tr}(u) = 0\}$. Then
    \[ H^1_0(\Omega)  \subset H_{\rm tr,0}^1(\Omega)\subset \widetilde{H}^1(\Omega) \subset H^1_{\rm tr}(\Omega). \]
\end{corollary}
\begin{proof}
    The only additional point to check is that $H_{\rm tr,0}^1(\Omega)\subset \widetilde{H}^1(\Omega)$.
     We consider a  function $\rho \in C^\infty_c(\mathbb{R}^d)$ such that
 \[ \rho \ge 0, \quad \big(|x| > 1 \implies \rho(x) = 0\big), \quad \int_{\mathbb{R}^d} \rho(x) \ {\rm d} x = 1. \]
 Let $u\in H^1_{\rm tr, 0}(\Omega)$. We first extend $u$  and $\partial_\theta u$ for all  $\theta\in {\mathbb{S}^{d-1}}$ by $0$ in $\mathbb{R}^d\setminus \Omega$. For all $n\in\mathbb{N}$, we define $u_n$ as the standard regularisation by convolution:
 \[
 \forall x\in \mathbb{R}^d,\ u_n(x) = n^d\int_{\mathbb{R}^d}u(y)\rho(n(x-y)){\rm d} y.
 \]
 Then $u_n$ converges to $u$ in $L^2(\mathbb{R}^d)$. Let $\theta\in {\mathbb{S}^{d-1}}$ be given. For all $x\in \mathbb{R}^d$ and $n\in\mathbb{N}$, we have, using the integration-by-parts formula \eqref{eq:intbypartshun} (since $\rho(n(x-\cdot))$ belongs to $C^0(\overline{\Omega})\cap H^1(\Omega)\subset H^1_{\rm tr}(\Omega)$) and the fact that ${\rm tr}(u) =0$ $\mu_\theta$-a.e. on $\partial\Omega$,
 \begin{multline*}
     \partial_\theta u_n(x) = n^d\int_{\mathbb{R}^d}u(y)\partial_\theta \rho(n(x-y)){\rm d} y
 = n^d\int_{\Omega}u(y)\partial_\theta \rho(n(x-y)){\rm d} y\\
 = n^d\int_{\Omega}\partial_\theta u(y) \rho(n(x-y)){\rm d} y = n^d\int_{\mathbb{R}^d}\partial_\theta u(y) \rho(n(x-y)){\rm d} y.
 \end{multline*}
 This proves that $\partial_\theta u_n$ converges to $\partial_\theta u$ in $L^2(\mathbb{R}^d)$, which concludes the proof that $u\in  \widetilde{H}^1(\Omega)$ since, also denoting by $u_n$ the restriction of $u_n$ to $\overline{\Omega}$, we have $u_n\in C^0(\overline{\Omega})\cap H^1(\Omega)$.
\end{proof}

\begin{remark}
    In the one-dimensional case, we show in Theorem \ref{thm:H1tilde=H1} that $\widetilde{H}^1(\Omega) = H^1_{\rm tr}(\Omega)$, and in Proposition \ref{Prophunzeqhunztr1d} that $H^1_0(\Omega) = H_{\rm tr,0}^1(\Omega)$.
    However, there exist situations in which $H^1_0(\Omega) \neq H_{\rm tr,0}^1(\Omega)$ (Section \ref{sec:cantcirc}) and $\widetilde{H}^1(\Omega) \neq H^1_{\rm tr}(\Omega)$ (Section \ref{sec:bicone}).
\end{remark}

Note that $H_{\rm tr,0}^1(\Omega) = \{ u \in H^1(\Omega) : \forall \theta \in \mathbb{S}^{d-1}, \ \gamma_\theta u = 0\}$. Hence, we have the following lemma.

\begin{lemma}
    For all $u \in H^1_{\rm tr, 0}(\Omega)$, then
    \begin{equation}\label{eq:poinhun}
        \Vert u\Vert_{L^2(\Omega)} 
        \le {\rm diam}(\Omega) \Vert \nabla u\Vert_{L^2(\Omega)}.
    \end{equation}
    Therefore the space $H_{\rm tr,0}^1(\Omega)$ is a Hilbert space, with the scalar product defined by 
    \[ \langle u,v\rangle := 
        \int_\Omega \nabla u(x)\cdot\nabla v(x){\ \rm d}x. \]
\end{lemma}

\begin{proof}
    For all $\theta\in\mathbb{S}^{d-1}$, $|\partial_\theta u| \le |\nabla u|$ a.e. in $\Omega$. Then \eqref{eq:poinhun} is a consequence of \eqref{eq:pointheta}, since it is satisfied provided that there exists $\theta \in \mathbb{S}^{d-1}$ with $\gamma_\theta u = 0$ $\mu_\theta$-a.e. on $\partial \Omega$. We then observe that $H_{\rm tr,0}^1(\Omega)\subset H_{\rm tr}^1(\Omega)$ is closed since ${\rm tr }:H_{\rm tr}^1(\Omega)\to L^2(\partial\Omega,(\mu_\theta)_{\theta\in\mathbb{S}^{d-1}})$ is continuous. The fact that $\langle \cdot, \cdot\rangle$ is a scalar product is a consequence of \eqref{eq:poinhun}.
\end{proof}

We finally extend Theorem \ref{thm:lebpoints} to all directions.
\begin{theorem} [Omnidirectional trace and directional Lebesgue points]\label{thm:omnilebpoints}
Let  $u\in H^1_{\rm tr}(\Omega)$. For all $\theta \in \mathbb{S}^{d-1}$ and $\varepsilon>0$ and for $\mu_\theta$ a.e. $z\in \partial_{\theta}\Omega$, we denote by
\[
u_{\theta,\varepsilon}(z) =  \frac 1 {\min(\varepsilon,\ell_{\minus\theta}(z))} \int_0^{\min(\varepsilon,\ell_{\minus\theta}(z))} u(z - s\theta) {\ \rm d}s.
\]
Then we have
\begin{equation}\label{eq:omnicvppbord}
 \hbox{For all $\theta \in \mathbb{S}^{d-1}$ and for $\mu_\theta$-a.e. }z\in \partial_{\theta}\Omega,\ \lim_{\varepsilon\to 0} u_{\theta,\varepsilon}(z) = {\rm tr}(u)(z)
\end{equation}
and
\begin{equation}\label{eq:omniintbord}
\max_{\theta\in \mathbb{S}^{d-1}} \int_{\partial_{\theta}\Omega} \Big( {\rm tr}(u)(z) - u_{\theta,\varepsilon}(z)\Big)^2 {\ \rm d}\mu_\theta(z) \le \varepsilon\ {\rm diam}(\Omega)\Vert\nabla u\Vert_{L^2(\Omega)}^2.
\end{equation}
 \end{theorem}
 \begin{proof}
     We apply Theorem \ref{thm:lebpoints} and we use $\Vert\partial_\theta u\Vert_{L^2(\Omega)}\le \Vert\nabla u\Vert_{L^2(\Omega)}$.
 \end{proof}
\subsection{The one-dimensional case} \label{sec:oned}

In this section, we consider a domain $\Omega \subset \mathbb R$ which is open and bounded. We first study the isolated points of the boundary of $\Omega$.

\begin{lemma}\label{lem:H1tr=H1}
    Let $\Omega \subset \mathbb R$ be bounded and open. Let $\mathcal{N}(\Omega)$ be the (countable) set of the isolated points of $\partial\Omega$, defined by
    \begin{equation}\label{eq:H1tr=H1}
        \mathcal{N}(\Omega) = \{z \in \partial\Omega, \; \exists r > 0, \; ]z-r, z+r[ \setminus \{z\}  \subset \Omega\}.
    \end{equation}
    Then denoting by $\Omega^* = \Omega \cup \mathcal N(\Omega)$, we have $H^1_{\rm tr}(\Omega) = \{u\in H^1(\Omega) : \exists v \in H^1(\Omega^*), \ u = v \hbox{ a.e. in } \Omega\}$.
\end{lemma}

\begin{proof}
    We first notice that $\mathcal{N}(\Omega) = \partial_1 \Omega \cap \partial_{\minus 1}\Omega$. Hence, for all $v \in H^1(\Omega^*)$, assimilating $v$ with a continuous representative in $\Omega^*$, we may define $u\in H^1(\Omega)$ by $u=v$ a.e. in $\Omega$. We then have
    \[ {\rm tr}(u)(z) = \begin{cases}
        \gamma_1(u)(z) & \text{ if } z \in \partial_1 \Omega \setminus \mathcal{N}(\Omega), \\
        \gamma_{\minus 1}(u)(z) & \text{ if } z \in \partial_{\minus 1}\Omega \setminus \mathcal{N}(\Omega),\\
        v(z) & \text{ if } z \in \mathcal{N}(\Omega),
    \end{cases} \]
    so that for all $\theta \in \mathbb{S}^0 = \{-1,1\}$, ${\rm tr}(u) = \gamma_\theta(u)$. Therefore, $u \in H^1_{\rm tr}(\Omega)$. Conversely, if $u \in H^1_{\rm tr}(\Omega)$, we have $\gamma_1(u)(z) = \gamma_{\minus 1}(u)(z) $ for all $z \in  \mathcal{N}(\Omega)$, which means that $u = v$ a.e. in $\Omega$, where $v \in H^1(\Omega^*)$.
\end{proof}

The goal is now to show that, for every bounded open set $\Omega \subset \mathbb{R}$, then any $u \in H^1_{\rm tr}(\Omega)$ is the limit of a sequence of continuous elements on $\overline \Omega$ belonging to $H^1(\Omega)$, which means that $H^1_{\rm tr}(\Omega) = \widetilde{H}^1(\Omega)$.

\begin{theorem}\label{thm:H1tilde=H1}
    Let $\Omega$ be an open bounded subset of $\mathbb R$. Then $\widetilde{H}^1(\Omega) = H^1_{\rm tr}(\Omega)$.
\end{theorem}

\begin{proof}
    We write $\Omega^*$ given by Lemma \ref{lem:H1tr=H1} as the countable union of open intervals $]a_i, b_i[$, $i \in I\subset\mathbb{N}$, such that all the closed intervals $[a_i, b_i]$ are disjoint.   We define $\alpha = \inf\{ a_i,i\in I\} - 1$ and $\beta = \sup\{ b_i,i\in I\} + 1$. Let $u \in H^1_{\rm tr}(\Omega)$. Using Lemma \ref{lem:H1tr=H1}, we may also denote by $u$ its continuous representative on all $[a_i, b_i]$. For all $M > 0$, let $T_M:\mathbb{R}\to \mathbb{R}$ be the truncation function $s\mapsto \min( \max(s,-M),M)$. Define $u_M:[\alpha,\beta]\to\mathbb{R}$ by
    \[ u_M(x) = \begin{cases}
        T_M(u(x)) & \text{ if } \hbox{there exists }i\in I\hbox{ s.t. } x \in [a_i, b_i],\\
        0 & \text{ otherwise}.
    \end{cases} \]
    Recall that truncation preserves $H^1$-membership with $\nabla u_M = \nabla u\cdot \mathbf{1}_{\{|u|<M\}}$ a.e. in $\Omega^*$ and that  
    \[
    ||u - u_M||_{H^1(\Omega)}^2
        \xrightarrow[M \to \infty]{} 0.\] 
    Let $\varepsilon > 0$ be given and let $M > 0$ such that $||u - u_M||_{H^1(\Omega)} \le \varepsilon$. For all $n \in \mathbb N$, let $\Omega^{(n)}$ be the union of a finite number of open intervals $]a_i, b_i[$ such that $\lambda(\Omega^* \setminus \Omega^{(n)}) \le 2^{-n}$. Then $[\alpha,\beta] \setminus \Omega^{(n)}$ is the union of a finite number of closed intervals $[c_j^n, d_j^n]$ for $j=1,\ldots,J$. Consider
    \[ v_n(x) = \begin{cases}
        u_M(x) & \text{ if } x \in \Omega^{(n)}, \\
        f_i^n(x) & \text { if } x \in [c_i^n, d_i^n],
    \end{cases} \]
    where $f_i^n : [c_i^n, d_i^n] \to \mathbb{R}$ is defined by
    $f_i^n(s) = u_M(c_i^n) (1-f(s)) + u_M(d_i^n) f(s)$ and $f$ is the generalised staircase function defined in  Lemma \ref{lem:escalierdiable} in Appendix \ref{sec:escalierdiable} with
    $f(c_i^n) = 0$ and $f(d_i^n)=1$ which is constant on all the open intervals of $\Omega^*$ included in $[c_i^n, d_i^n]$. Then $|v_n(x) - u_M(x)| \le 2M$ for $x \in [c_i^n, d_i^n]$ and
    \begin{align*}
        ||v_n - u_M||_{H^1(\Omega)}^2 & = ||v_n - u_M||_{L^2(\Omega)}^2+||\nabla v_n -\nabla u_M||_{L^2(\Omega)}^2\\
        & \hspace{1cm} \le 4 M^2 2^{-n} + ||\nabla u_M||_{L^2(\Omega^* \setminus \Omega^{(n)})}^2 \xrightarrow[n \to \infty]{} 0.
    \end{align*}
    Moreover, $v_n$ is continuous on $[\alpha,\beta]$ and therefore on $\overline \Omega$. Finally, there exists $N \in \mathbb N$ such that for all $n \ge N$,
    \[ ||v_n - u||_{H^1(\Omega)} \le ||v_n - u_M||_{H^1(\Omega)} + ||u_M - u||_{H^1(\Omega)} \le 2 \varepsilon. \]
    We conclude that $u \in \overline{H^1(\Omega) \cap \mathcal C(\overline \Omega)}$, so that the claimed equality holds using Corollary \ref{cor:inclusionshuntr}.
\end{proof}

\begin{proposition}\label{Prophunzeqhunztr1d}
    Let $\Omega$ be an open bounded subset of $\mathbb R$. Then $H^1_0(\Omega) = H^1_{\rm tr, 0}(\Omega)$.
\end{proposition}

\begin{proof}
Note that, in this proof, we do not suppress the isolated points of the boundary of $\Omega$, in contrast to the proof of Theorem \ref{thm:H1tilde=H1}.
    Let $u \in H^1_{\rm tr, 0}(\Omega)$. We again write $\Omega$ as the countable union of disjoint open intervals $]a_i, b_i[$, $i \in I\subset\mathbb{N}$.
    Let $\varepsilon > 0$  be given. As in Theorem \ref{thm:H1tilde=H1}, for all $n \in \mathbb N$, set $\Omega^{(n)}$ be the finite union of the open intervals $]a_i, b_i[$ for $i \in I_n\subset I$, such that $\lambda(\Omega \setminus \Omega^{(n)}) \le 2^{-n}$. Consider $v_n:\mathbb{R}\to \mathbb{R}$ defined by
    \[ v_n(x) = \begin{cases}
        u(x) & \text{ if } x \in \Omega^{(n)}, \\
        0 & \text { otherwise}.
    \end{cases} \]
    Since we have
    \[ ||v_n - u||_{H^1(\Omega)}^2 = ||u||_{H^1(\Omega\setminus \Omega^{(n)})}^2 \xrightarrow[n \to \infty]{} 0, \]
    there exists $N \in \mathbb N$ such that for all $n \ge N$, $||v_n - u||_{H^1(\Omega)} \le \varepsilon$. For such $n$ and for all $i \in I_n$, we have $v_n \in H^1_{0}(]a_i, b_i[)$, so there exists $w_{n,i} \in C^\infty_c(]a_i, b_i[)$ such that
    \[ \Vert v_n - w_{n,i} \Vert_{H^1(]a_i, b_i[)}^2 \le \frac{\varepsilon^2}{\#I_n}. \]
    Hence, if we set $w_n = w_{n,i}$ on $]a_i, b_i[$ for all $i\in I_n$ and $0$ elsewhere, we obtain $\Vert v_n - w_n \Vert_{H^1(\Omega^{(n)})}^2 \le \varepsilon^2$ with $w_n\in C^\infty_c(\Omega)$.
     Therefore,
    \[ \Vert w_n - u \Vert_{H^1(\Omega)} \le \Vert w_n - v_n \Vert_{H^1(\Omega)} + \Vert v_n - u \Vert_{H^1(\Omega)} \le 2 \varepsilon, \]
    so $u \in H^1_0(\Omega)$, and the claimed equality holds using Corollary \ref{cor:inclusionshuntr}.
\end{proof}

\subsection{Definition of the spaces  $H^{1/2}(\partial\Omega)$ and $H^{-1/2}(\partial\Omega)$ for general domains}\label{sec:proptrace}

Owing to the Hilbert structure of $H^1_{\rm tr}(\Omega)$, we can define the Hilbert space which contains the traces of elements of $H^1_{\rm tr}(\Omega)$. We thus define
\[ H^{1/2}(\partial\Omega) = \{ {\rm tr}(u) : u\in H^1_{\rm tr}(\Omega)\}. \]

\begin{proposition}\label{prop:defpsi}
    Let $g \in H^{1/2}(\partial\Omega)$. There exists a unique $u \in H^1_{\rm tr}(\Omega)$ such that ${\rm tr}(u) = g$ and for all $w \in H^1_{\rm tr, 0}(\Omega)$,
    \[ \int_\Omega \nabla u \cdot \nabla w {\ \rm d} x = 0. \]
    We denote $u = \Psi(g)$, and the map $\Psi : H^{1/2}(\partial\Omega) \to H^1_{\rm tr}(\Omega)$ is injective.
\end{proposition}

\begin{proof}
    Let $g \in H^{1/2}(\partial\Omega)$. There exists $v \in H^1_{\rm tr}(\Omega)$ such that ${\rm tr}(v) = g$. By the Lax-Milgram theorem, there exists a unique $\widetilde v \in H^1_{\rm tr, 0}(\Omega)$ such that for all $w \in H^1_{\rm tr, 0}(\Omega)$,
    \[ \int_\Omega \nabla \widetilde v \cdot \nabla w {\ \rm d}x
        = \int_\Omega \nabla v \cdot \nabla w {\ \rm d}x. \]
    Considering $u = v - \widetilde v$, we obtain ${\rm tr}(u) = {\rm tr}(v) = g$, and for all $w \in H^1_{\rm tr, 0}(\Omega)$,
    \[ \int_\Omega \nabla u \cdot \nabla w {\ \rm d}x
        = \int_\Omega \nabla (v-\widetilde v) \cdot \nabla w {\ \rm d}x
        = 0. \]
    This proves the existence. To show the uniqueness, let $\widetilde u \in H^1_{\rm tr}(\Omega)$ satisfying the two properties. Then $u - \widetilde u \in H^1_{\rm tr, 0}(\Omega)$, so that $\Vert \nabla(u-\tilde u) \Vert_{L^2(\Omega)} = 0$. Hence, $u = \widetilde u$.
\end{proof}

\begin{corollary}\label{cor:hhalfhilbert}
    The space $H^{1/2}(\partial\Omega)$ is a Hilbert space with the scalar product
    \[ \langle g, \widetilde{g}\rangle_{H^{1/2}}
        = \langle \Psi(g), \Psi(\widetilde{g})\rangle_{H^{1}}. \]
\end{corollary}

\begin{proof}
    This is a consequence of the fact that $\Psi(H^{1/2}(\partial\Omega))$ is a closed subset of $H^1_{\rm tr}(\Omega)$, and that the map $\Psi : H^{1/2}(\partial\Omega) \to \Psi(H^{1/2}(\partial\Omega))$ is bijective.
\end{proof}

We can then define the following space of harmonic functions:
\[ D_\nu(\Omega) = \left\{u\in H^1(\Omega) : \forall w \in H^1_{\rm tr,0}(\Omega), \ \int_\Omega \nabla u \cdot \nabla w {\ \rm d} x = 0\right\}. \]
By definition, for all $g \in H^{1/2}(\partial\Omega)$, $\Psi(g)$ is harmonic.

\begin{definition}[Generalized normal derivative and Dirichlet-to-Neumann operator]\label{def:gennormalder}
    For all $u \in D_\nu(\Omega)$, let $\partial_\nu u$ the continuous linear form defined on $H^{1/2}(\partial\Omega)$ by
    \[ \langle \partial_\nu u,g\rangle_{H^{-1/2},H^{1/2}} 
        = \int_\Omega \nabla u\cdot\nabla \Psi(g) {\ \rm d} x. \]
    We call $\partial_\nu u$ the generalized normal derivative of $u$. It is an element of the space $H^{-1/2}(\partial\Omega)$ of all continuous linear forms on $H^{1/2}(\partial\Omega)$. The Dirichlet-to-Neumann operator is then defined by
    \[ D : \begin{cases}
        H^{1/2}(\partial\Omega) & \to \qquad H^{-1/2}(\partial\Omega) \\
        \qquad g & \mapsto \quad D(g) = \partial_\nu \Psi(g).
    \end{cases} \]
\end{definition}

We notice that $D$ is the standard isomorphism between $H^{1/2}(\partial\Omega)$ and its topological dual space.

\section{Examples and applications}

 \subsection{Preliminary: a linear form supported by a fractal domain}\label{sec:fractal}

  The construction of the domain $\Omega_{\mathcal{C}}$ (see Figure \ref{fig:monocone}) is inspired by results of \cite{mazya} on sets with positive capacity. Following \cite{falconer1990fractal}, we consider a domain whose boundary includes a fractal subset with positive capacity, but zero 1D-Lebesgue measure. Then, accounting for some results from \cite{ziemer1989weakly} concerning Poincar\'e inequalities linked with the notion of capacity, we construct a continuous linear functional on $H^1(\Omega_{\mathcal{C}})$ that equals the integral of the Sobolev trace of the quasi-continuous representative of $u$ on $C_{1/3} \times \{0\}$ with respect to the Hausdorff measure of dimension $\log 2/\log 3$. Nevertheless, for the sake of completeness, we provide a short, self-contained proof that does not rely on results from potential theory or the theory of sets of positive capacity.

     Let us first define the middle third Cantor set, denoted by $\mathcal{C}_{1/3}\subset [0,1]$. We define the values $a_m,b_m,c_m,d_m$ for all positive integer $m$ by
     \begin{itemize}
         \item $a_1=0$, $b_1 = 1$;
         \item For all integer $n\ge 0$, and any $m=2^n,\ldots,2^{n+1}-1$, we define $c_m = \frac 2 3 a_m+\frac 1 3 b_m$ and $d_m = \frac 1 3 a_m+\frac 2 3 b_m$; then we define $a_{2m} = a_m$, $b_{2m} = c_m$, $a_{2m+1} = d_m$, $b_{2m+1} = b_m$.
     \end{itemize}
     Hence we observe that $n=0$ provides  $c_1 = \frac 1 3$, $d_1 = \frac 2 3$,  $a_2 = a_1$, $b_2 = c_1$, $a_3 = d_1$ and $b_3 = b_1$.
     We define the fractal Cantor set $\mathcal{C}_{1/3}$ by
     \[ \mathcal{C}_{1/3} = [0,1]\setminus \bigcup_{m\ge 1} ]c_m,d_m[. \]
     The 1D-Lebesgue measure of $\mathcal{C}_{1/3}$ is then equal to 0.
     
    For all $a\in [0,1]$, we define the cone with vertex $(a,0)$, angle $\pi/2$ and vertical height equal to $1$:
    \[ K_a = \{(x_1,x_2)\in \mathbb{R}^2 : \ 0 < x_2 < 1\hbox{ and } |x_1-a|<x_2\}. \]
    We define the open set $\Omega_{\mathcal{C}}$ (see Figure \ref{fig:monocone}) by
    \begin{equation}\label{eq:defomegac}
     \Omega_{\mathcal{C}} = \bigcup_{a\in \mathcal{C}_{1/3}} K_a.   
    \end{equation} 
\begin{lemma}\label{lem:monocone}
 Let $\Omega$ be an open bounded domain of $\mathbb{R}^2$ such that $\Omega_{\mathcal{C}}\subset \Omega$ where $\Omega_{\mathcal{C}}$ is defined by \eqref{eq:defomegac}.
For all $n\in\mathbb{N}$, we define $y_n = 3^{-n} /2$ and $\mathcal{D}_n = \{x\in\mathbb{R} : (x,y_n)\in\Omega_{\mathcal{C}}\}$.
    
    We define the linear form $\nu_n\in H^1(\Omega)'$ by
    \begin{equation}
\forall u\in  H^1(\Omega),\ \nu_n(u) = \frac 1 {|\mathcal{D}_n|}\int_{\mathcal{D}_n} u(x,y_n) {\ \rm d}x,
\label{def-nun}
\end{equation}
    where, by abuse of notation, we also denote by $u$ the standard trace of $u$ on any line $\mathcal{D}_n$.
    Then the sequence $(\nu_n)_{n\in\mathbb{N}}$ converges to a linear form denoted by $\overline{\nu}\in H^1(\Omega)'$. This linear form is such that, if $u\in H^1(\Omega)$ is equal a.e. in $\Omega_{\mathcal{C}}$ to a constant $A$, then $\overline{\nu}(u) = A$, and for all $u\in H^1(\Omega)$ such that there exists $V$, open set of $\mathbb{R}^2$, with $\mathcal{C}_{1/3}\times \{0\}\subset V$ and $u(x) = 0$ for a.e. $x\in V\cap \Omega_{\mathcal{C}}$, then $\overline{\nu}(u) =0$.
\end{lemma}

    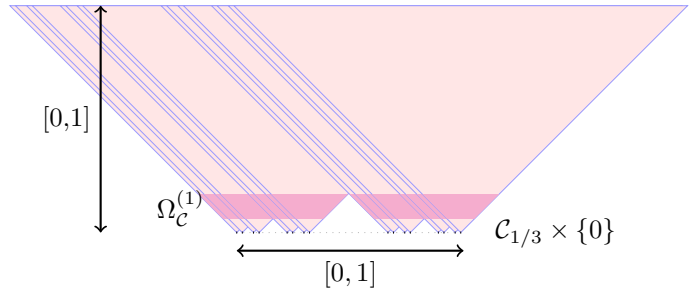
\begin{figure}[!ht]
    \begin{center}
    \begin{tikzpicture}[scale=6] 
    
        \def\h{0.5}      
        \def\ylow{0.03} 
        \def\yhigh{0.084} 
    
        \def\cantorL{0.5} 
    
        \tikzset{
            upperCone/.style={fill=red!10, draw=blue!40, opacity=0.4, ultra thin},
            midZone/.style={fill=magenta!40, opacity=0.7}, 
            cantorPoint/.style={black, thick}
        }
    
        \draw[dotted, gray!50] (0,0) -- (\cantorL, 0);
    
        \foreach \i in {0,...,15} {
            \pgfmathsetmacro{\p}{
                ((mod(floor(\i/8),2)*18 + mod(floor(\i/4),2)*6 + mod(floor(\i/2),2)*2 + mod(\i,2)*0.666)/27) * \cantorL
            }
    
            \coordinate (S) at (\p, 0);
    
            \fill[upperCone] (S) -- (\p - \h, \h) -- (\p + \h, \h) -- cycle;
            \begin{scope}
                \clip (-0.25, \ylow) rectangle (0.75, \yhigh);
                \fill[midZone] (S) -- (\p - \h, \h) -- (\p + \h, \h) -- cycle;
            \end{scope}
    
    
            \draw[cantorPoint] (\p-0.001, 0) -- (\p+0.001, 0);
        }
    
        \draw[<->, thick] (0, -0.04) -- node[below] { $[0,1]$} (\cantorL, -0.04);
        \draw[<->, thick] (-0.3, 0) -- node[left] {[0,1]} (-0.3, \h);
    
        \node[right] at (\cantorL+0.05, 0) {$\mathcal{C}_{1/3}\times\{0\}$};
        \node[left] at (-.05, \yhigh-.03) {$\Omega_{\mathcal{C}}^{(1)}$};
    \end{tikzpicture}
    \end{center}
    \caption{The domains $\Omega_{\mathcal{C}}$ and $\Omega_{\mathcal{C}}^{(1)}$.\label{fig:monocone}}
    \end{figure}

\begin{proof}
    Note that
    \[ \mathcal{D}_n = \bigcup_{m=2^n}^{2^{n+1}-1} I_m^{(n)}\hbox{ with }I_m^{(n)}:= \ ]a_m-y_n, b_m+y_n[, \] 
    where the length of $I_m^{(n)}$ is equal to $2\cdot 3^{-n}$, hence the measure of $\mathcal{D}_n$ is equal to $2 (2/3)^{n}$ 
     and we also have
    \[ \nu_n(u) = 2^{-n}\sum_{m=2^n}^{2^{n+1}-1} \frac 1 {|I_m^{(n)}|}\int_{I_m^{(n)}} u(x,y_n){\ \rm d}x. \] 
    
    In order to prove the convergence of the sequence $(\nu_n)_n$ in $H^1(\Omega)'$, we compute $\nu_{n+1}(u) - \nu_n(u)$. 
    We denote by $J_{2m}^{(n)}:=]a_m-y_n, \frac {a_m+b_m} 2[$ and $J_{2m+1}^{(n)}:=]\frac {a_m+b_m} 2, b_m+y_n[$.
    We observe that
    \[ \frac 2 {|I_m^{(n)}|}\int_{I_m^{(n)}} u(x,y_n){\ \rm d}x 
        = \frac 1 {|J_{2m}^{(n)}|}\int_{J_{2m}^{(n)}} u(x,y_n){\ \rm d}x 
        + \frac 1 {|J_{2m+1}^{(n)}|}\int_{J _{2m+1}^{(n)}} u(x,y_n){\ \rm d}x. \]

    We have
    \[ \nu_{n+1}(u) - \nu_n(u)
        = 2^{-(n+1)}\sum_{m=2^n}^{2^{n+1}-1} (T_{2m}^{(n)} + T_{2m+1}^{(n)}), \]
    with
    \[ T_{2m}^{(n)} = \frac 1 {|I_{2m}^{(n+1)}|}\int_{I_{2m}^{(n+1)}} u(x,y_{n+1}){\ \rm d}x
        - \frac 1 {|J_{2m}^{(n)}|}\int_{J_{2m}^{(n)}} u(x,y_n){\ \rm d}x, \]
    and
    \[ T_{2m+1}^{(n)} = \frac 1 {|I_{2m+1}^{(n+1)}|}\int_{I_{2m+1}^{(n+1)}} u(x,y_{n+1}){\ \rm d}x - \frac 1 {|J_{2m+1}^{(n)}|}\int_{J_{2m+1}^{(n)}} u(x,y_n){\ \rm d}x. \]
    Let us find a bound for $T_{2m}^{(n)}$, using the gradient of $u$ (the bound for $T_{2m+1}^{(n)}$ will then be identical).
    We have, owing the change of variable $x \to a_{2m}-y_{n+1}+ 2(x- a_m+y_n)/3$ in the first integral of $T_{2m}^{(n)}$, 
    \[ T_{2m}^{(n)} = \frac 1 {|J_{2m}^{(n)}|}\int_{J_{2m}^{(n)}} \left(u\left(a_{2m}-y_{n+1}+ \frac{2(x- a_m+y_n)}{3},y_{n+1}\right) - u(x,y_n)\right){\ \rm d}x. \]
    We then write
    \[ u(a_{2m}-y_{n+1}+ 2(x- a_m+y_n)/3,y_{n+1}) - u(x,y_n) 
        = \int_{[y_{n+1},y_n]} \nabla u( \beta(x) y + \zeta(x),y) \cdot \frac {\bm v}{|{\bm v}|}{\ \rm d}y, \]
    where ${\bm v} = (a_{2m}-y_{n+1}+ 2(x- a_m+y_n)/3 - x,y_{n+1} - y_n)$, $\beta(x) y_{n+1} + \zeta(x) = a_{2m}-y_{n+1}+ 2(x- a_m+y_n)/3$ and $\beta(x) y_{n} + \zeta(x) = x$. 
    We make the change of variable $(x,y)\to (\beta(x) y + \zeta(x),y)$, whose Jacobian is equal to $J(x,y) = \frac {\frac  1 3 y + \frac 2 3 y_n - y_{n+1}}{y_n - y_{n+1}}$. Since $\frac 2 3 \le J(x,y)\le 1$, we get that
    \[ |T_{2m}^{(n)}| \le  \frac 1 {|J_{2m}^{(n)}|}\int_{\Omega_{2m}^{(n)} } \vert \nabla u(x,y)\vert {\ \rm d}x{\ \rm d}y \hbox{ and similarly }
    |T_{2m+1}^{(n)}| \le \frac 1 {|J_{2m+1}^{(n)}|}\int_{\Omega_{2m+1}^{(n)} } \vert \nabla u(x,y)\vert {\ \rm d}x{\ \rm d}y, \]
    where
    \[ \Omega_{2m}^{(n)}= \Omega_{\mathcal{C}} \cap (J_{2m}^{(n)}\times ]y_{n+1},y_n[)\hbox{ and }\Omega_{2m+1}^{(n)}= \Omega_{\mathcal{C}} \cap (J_{2m+1}^{(n)}\times ]y_{n+1},y_n[). \]
    Therefore, denoting by 
    \[ \Omega_{\mathcal{C}}^{(n)} = \{ (x,y)\in\Omega_{\mathcal{C}}, y_{n+1} < y < y_n\} = \bigcup_{m=2^n}^{2^{n+1}-1} (\Omega_{2m}^{(n)}\cup\Omega_{2m+1}^{(n)}), \]
    we find that 
    \[ |\nu_{n+1}(u) - \nu_n(u)| \le  \frac {2^{-n}} {3^{-n}} \int_{\Omega_{\mathcal{C}}^{(n)} } \vert \nabla u(x,y)\vert {\ \rm d}x{\ \rm d}y. \]
    Owing to the Cauchy-Schwarz inequality, and to the fact that $\lambda^2(\Omega_{\mathcal{C}}^{(n)}) \le 2^{n+1} 3^{-2n}$, we get
    \[ |\nu_{n+1}(u) - \nu_n(u)|^2 \le  \frac {2^{-2n}} {3^{-2n}} 2^{n+1} 3^{-2n}\int_{\Omega_{\mathcal{C}}^{(n)} } \vert \nabla u(x,y)\vert^2 {\ \rm d}x{\ \rm d}y, \]
    which implies 
    \[ |\nu_{n+1}(u) - \nu_n(u)| \le 2^{(1-n)/2} \Vert  u\Vert_{H^1(\Omega)}. \]
    This first shows that the series $\sum_n (\nu_{n+1} - \nu_n)$ is absolutely convergent in $H^1(\Omega)'$, which, using that $\nu_0\in H^1(\Omega)'$, allows us to define $\overline{\nu}\in H^1(\Omega)'$ by
    \[ \overline{\nu} = \lim_{n\to\infty}\nu_n. \]
    We then notice that, if $u\in H^1(\Omega)$ is equal a.e. in $\Omega_{\mathcal{C}}$ to a constant $A$, then $\nu_n(u) = A$ for all $n\ge 0$, which yields $\overline{\nu}(u) = A$. 
    
    Let $u\in H^1(\Omega)$ such that there exists $V$, open set of $\mathbb{R}^2$, with $\mathcal{C}_{1/3}\times \{0\}\subset V$ and $u(x) = 0$ for a.e. $x\in V\cap \Omega_{\mathcal{C}}$. We notice that the distance $a$ between the compact set $\mathcal{C}_{1/3}\times\{0\}$ and the boundary of $V$ is positive. Hence there exists $n_0\in \mathbb{N}$ such that, for all $n\ge n_0$, $\mathcal{D}_n\times\{y_n\}\subset V$. For such $n$, we get $\nu_n(u) = 0$, which shows that $\overline{\nu}(u) = 0$.
\end{proof}

\begin{corollary}[Positive capacity] Let  $\Omega = B((\frac 1 2,0),2)\subset \mathbb{R}^2$ (this open ball satisfies $\overline{\Omega_{\mathcal{C}}}\subset \Omega$, where $\Omega_{\mathcal{C}}$ is defined by \eqref{eq:defomegac}, see Figure \ref{fig:cantcirc}). Denote $K= \mathcal{C}_{1/3}\times\{0\}$ and define the capacity of $K$ relative to $\Omega$, following for example \cite{edmundsevans1987,boc1996exist},
 \begin{equation}\label{eq:defcap}
   {\rm Cap}(K,\Omega) = \inf\{ \Vert\nabla u\Vert_{L^2(\Omega)}^2, \ u\in C^\infty_c(\Omega), u\ge \chi_K\}.   
 \end{equation}
 Then ${\rm Cap}(K,\Omega) >0$.
\end{corollary}

\begin{proof}
Let $n \in \mathbb{N}^\star$.
Equation \eqref{def-nun} defines an element of $C^0(\overline{\Omega})'$ and an element of $H^{-1}(\Omega)$, that is to say
\begin{align*}
& \langle \nu_n,u\rangle_{C^0(\overline{\Omega})',C^0(\overline{\Omega})}=\nu_n(u), \textrm{ for } u \in C^0(\overline{\Omega}),
\\
& \langle \nu_n,u\rangle_{H^{-1}(\Omega),H^1_0(\Omega)}=\nu_n(u), \textrm{ for } u \in H^1_0(\Omega).
\end{align*}
As an element of $C^0(\overline{\Omega})'$, $\nu_n$ is a measure of mass $1$.
Since $C^0(\overline{\Omega})$ is a separable Banach space, $\nu_n$ converges up to a subsequence, as $n \to +\infty$, $C^0(\overline \Omega)'$-$\star$-weakly to some positive Radon measure $\nu$ of mass $1$.
Then $\nu$ is supported by $K$, since for $u \in C^0(\overline{\Omega})$ with $u=0$ in $K$, using the uniform continuity of $u$ in $\overline{\Omega}$, one proves that $\nu_n(u) \to 0$ as $n \to \infty$.

The proof of Lemma \ref{lem:monocone} shows that $\nu_n$
converges in $H^{-1}(\Omega)$ to  $\overline \nu$ of $H^{-1}(\Omega)$, as $n \to +\infty$. Since the support of $\nu$ is reduced to $K$, we get that $\nu$ is uniquely defined by $\overline \nu$, which implies that $\nu_n$ converge to $\nu$ without extraction of a subsequence.

If $u \in C^0(\overline{\Omega})\cap H^1_0(\Omega)$, $\langle \nu,u\rangle_{C^0(\overline{\Omega})',C^0(\overline{\Omega})}=\langle \overline \nu,u\rangle_{H^{-1}(\Omega),H^1_0(\Omega)}$ (this can be proven, for instance, using the fact that if $u \in C^0(\overline{\Omega}) \cap H^1_0(\Omega)$ there exist a sequence $(u_n)_{n \in \mathbb{N}}$ of $C^\infty_c(\Omega)$ converging uniformly and in $H^1_0(\Omega)$ to $u$).

\medskip

Let $M=\Vert \overline \nu \Vert_{H^{-1}(\Omega)}$ (taking $\Vert u \Vert_{H^1_0(\Omega)}^2=\int_\Omega |\nabla u(x) |^2 dx$).
Using \eqref{eq:defcap}, let us prove that
\begin{equation}
 {\rm Cap}(K,\Omega) \ge \frac 1 {M^2}.
\label{cap-K}
\end{equation}
Indeed, let $u \in C^\infty_c(\Omega)$, $u \ge \chi_K$.
Using the fact that $\nu$ is a measure of mass $1$ supported by $K$,
\[
1 \le \langle \nu,u\rangle_{C^0(\overline{\Omega})',C^0(\overline{\Omega})} = \langle \overline \nu,u\rangle_{H^{-1}(\Omega),H^1_0(\Omega)} \le M \Vert u \Vert_{H^1_0(\Omega)}.
\]
This gives $\Vert u \Vert_{H^1_0(\Omega)} \ge 1/M$ and then \eqref{cap-K}.
\end{proof}

 \subsection{Case in which $H^1_0(\Omega)\neq H_{\rm tr,0}^1(\Omega)$}\label{sec:cantcirc}

\begin{figure}[!ht]
\begin{center}
\begin{tikzpicture}[scale=3]
    \draw[draw=black, fill=red!10, thick] (0.25,0) circle (1);
    
        \def\h{0.5}      
        \def\ylow{0.03} 
        \def\yhigh{0.084} 
    
        \def\cantorL{0.5} 
    
        \tikzset{
            upperCone/.style={fill=red!10, draw=blue!40, opacity=0.4, ultra thin},
            cantorPoint/.style={black, thick}
        }
    
        \draw[dotted, gray!50] (0,0) -- (\cantorL, 0);
    
        \foreach \i in {0,...,15} {
            \pgfmathsetmacro{\p}{
                ((mod(floor(\i/8),2)*18 + mod(floor(\i/4),2)*6 + mod(floor(\i/2),2)*2 + mod(\i,2)*0.666)/27) * \cantorL
            }
    
            \coordinate (S) at (\p, 0);
    
            \fill[upperCone] (S) -- (\p - \h, \h) -- (\p + \h, \h) -- cycle;

            \draw[cantorPoint] (\p-0.005, 0) -- (\p+0.005, 0);
        }
\end{tikzpicture}
\end{center}
    \caption{Domain $\Omega$ with  $H^1_0(\Omega)\neq H_{\rm tr,0}^1(\Omega)$.} \label{fig:cantcirc}
\end{figure}
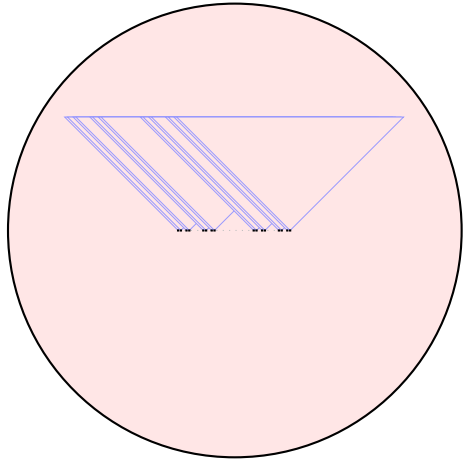

We consider the same example as \cite[Example 5.3]{are2024perron}. We let $d = 2$, and 
$\Omega = B((\frac 1 2,0),2)\setminus (\mathcal{C}_{1/3}\times\{0\})$ (see Figure \ref{fig:cantcirc}), which implies that $\partial\Omega = (\mathcal{C}_{1/3}\times\{0\})\cup \partial B((\frac 1 2,0),2)$ 
and that $\Omega_{\mathcal{C}}\subset \Omega$ (where the middle third Cantor set $\mathcal{C}_{1/3}$ and  $\Omega_{\mathcal{C}}$ are defined in Section \ref{sec:fractal}). We now give an example of a function $u \in H^1_{\mathrm{tr},0}(\Omega) \setminus H^1_0(\Omega)$.

\medskip

We first remark that, since the 1D-Lebesgue measure of the projection of $\mathcal{C}_{1/3}\times\{0\}$ onto $H_\theta$ for all $\theta\in\mathbb{S}^{d-1}$ is equal to $0$, then $u\in H_{\rm tr,0}^1(\Omega)$ if and only if $u\in H^1_0(B((\frac 1 2,0),2))$. Let $u$ be chosen such that $\overline{\nu}(u)\neq 0$ (for example, we assume that $u(x) = 1$ for a.e. $x\in \Omega_{\mathcal{C}}$, then, from Lemma \ref{lem:monocone}, we get that $\overline{\nu}(u)=1$). Assume that $(v_n)_n$ is a sequence of elements of $H^1_0(\Omega)$. For each $n$, the complement of support$(v_n)$ in $\mathbb{R}^2$ is open and contains the boundary of $\Omega$, including $\mathcal{C}_{1/3}\times\{0\}$. From Lemma \ref{lem:monocone}, we deduce that $\overline{\nu}(v_n)=0$. Since $\overline{\nu}$ is continuous, $u$ cannot be the limit in $H^1(\Omega)$ of the sequence $(v_n)_n$.

\subsection{Case  in which $\widetilde{H}^1(\Omega)\neq H^1_{\rm tr}(\Omega)$}\label{sec:bicone}

\begin{lemma}\label{lem:bicone}
    There exists a bounded open set $\Omega\subset\mathbb{R}^2$ such that $H^1_{\rm tr}(\Omega)\setminus \widetilde{H}^1(\Omega)$ is not empty.
\end{lemma}

        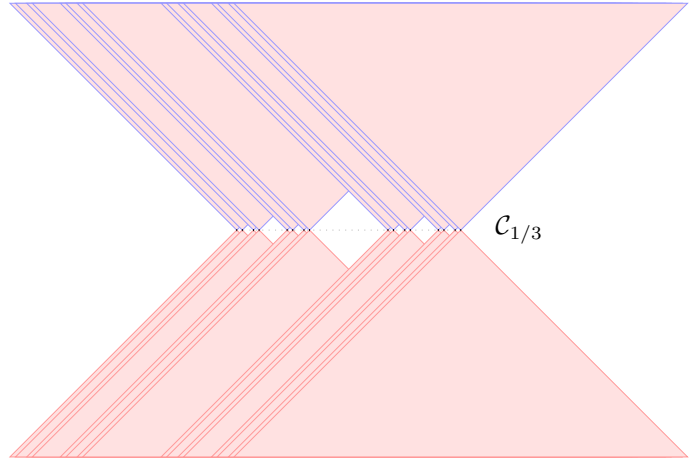
\begin{figure}[!ht]
    \begin{center}
    \begin{tikzpicture}[scale=6] 
    
        \def\h{0.5}      
        \def\ylow{0.03} 
        \def\yhigh{0.084} 
    
        \def\cantorL{0.5} 
    
        \tikzset{
            upperCone/.style={fill=red!12, draw=blue!40, opacity=0.4, ultra thin},
            lowerCone/.style={fill=red!12, draw=red!40, opacity=0.4, ultra thin},
            cantorPoint/.style={black, thick}
        }
    
        \draw[dotted, gray!50] (0,0) -- (\cantorL, 0);
    
        \foreach \i in {0,...,15} {
            \pgfmathsetmacro{\p}{
                ((mod(floor(\i/8),2)*18 + mod(floor(\i/4),2)*6 + mod(floor(\i/2),2)*2 + mod(\i,2)*0.666)/27) * \cantorL
            }
    
            \coordinate (S) at (\p, 0);
    
            \fill[upperCone] (S) -- (\p - \h, \h) -- (\p + \h, \h) -- cycle;
    
            \fill[lowerCone] (S) -- (\p - \h, -\h) -- (\p + \h, -\h) -- cycle;
    
            \draw[cantorPoint] (\p-0.001, 0) -- (\p+0.001, 0);
        }
    
        \node[right] at (\cantorL+0.05, 0) {$\mathcal{C}_{1/3}$};
    \end{tikzpicture}
    \end{center}
    \caption{The domain $\Omega$ with $\widetilde{H}^1(\Omega)\neq H^1_{\rm tr}(\Omega)$.\label{fig:bicone}}
    \end{figure}

\begin{proof}
    We let $d = 2$. We define the open set $\Omega$ (see Figure \ref{fig:bicone}) by
    \[ \Omega = \Omega_{\mathcal{C}}\cup \psi(\Omega_{\mathcal{C}}), \]
    where $\Omega_{\mathcal{C}}$ is defined in Section \ref{sec:fractal}, and $\psi:(x,y)\mapsto(x,-y)$ is the symmetry with respect to the second coordinate. Using the linear form $\overline{\nu}$ provided by Lemma \ref{lem:monocone}, we denote by $\overline{\nu}_\psi$ the linear form $u\mapsto \overline{\nu}(u\circ \psi)$.
    
    We claim that $(\overline{\nu}- \overline{\nu}_\psi)(u) = 0$ for all $u\in C^0(\overline{\Omega})\cap H^1(\Omega)$. Indeed, we have
    \[ (\overline{\nu}- \overline{\nu}_\psi)(u) = \lim_{n\to\infty} \frac 1 {|\mathcal{D}_n|}\int_{\mathcal{D}_n} (u(x,y_n)-u(x,-y_n)){\ \rm d}x. \]
    Since $u\in C^0(\overline{\Omega})$ and $\overline{\Omega}$ is compact, for all $\varepsilon>0$ there exists $\delta>0$ with $|u(x,y) - u(x',y')|\le \varepsilon$ if $|(x,y)-(x',y')|\le \delta$.  Therefore, for $n$ such that $y_n\le \delta/2$,
    \[ \big| \frac 1 {|\mathcal{D}_n|}\int_{\mathcal{D}_n}  (u(x,y_n)-u(x,-y_n)){\ \rm d}x \big|\le \varepsilon. \]
    Hence, for a sequence $(u_k) \in C^0(\overline{\Omega})\cap H^1(\Omega)$ converging in $H^1(\Omega)$ to some $u\in H^1(\Omega)$, since $\overline{\nu}- \overline{\nu}_\psi$ is a continuous linear form, we get $(\overline{\nu}- \overline{\nu}_\psi)(u)=0$. \medskip
    
    We define the function $\widehat{u}$ by $\widehat{u}(x,y) = 1$ for $y>0$ and $-1$ for $y<0$. It satisfies $(\overline{\nu}- \overline{\nu}_\psi)(\widehat{u}) = 1 - (-1) = 2$, which implies that $\widehat{u}\notin \widetilde{H}^1(\Omega)$. We may write that 
    $ \partial\Omega = \partial\Omega_{\mathcal{C}}\cup \partial\psi(\Omega_{\mathcal{C}})$, with $\partial\Omega_{\mathcal{C}}\cap \partial\psi(\Omega_{\mathcal{C}}) = \mathcal{C}_{1/3}\times\{0\}$. Since  $\widehat{u} = 1$ on $\Omega_{\mathcal{C}}$, we get that $\gamma_\theta \widehat{u} = 1$ $\mu_\theta$-a.e. on $\partial\Omega_{\mathcal{C}}$ for all $\theta\in \mathbb{S}^{d-1}$. Similarly, we have $\gamma_\theta \widehat{u} = -1$ $\mu_\theta$-a.e. on $\partial\psi(\Omega_{\mathcal{C}})$ for all $\theta\in \mathbb{S}^{d-1}$. Since $\mathcal{C}_{1/3}\times\{0\}$ is $\mu_\theta$-negligible for  all $\theta\in \mathbb{S}^{d-1}$, we get that $\widehat{u}\in H^1_{\rm tr}(\Omega)$ with ${\rm tr}(\widehat{u}) = 1$ a.e. on $\partial\Omega_{\mathcal{C}}$ and $-1$ on $\partial\psi(\Omega_{\mathcal{C}})$.
\end{proof}

\medskip

Let us now consider the following non-homogeneous Dirichlet problems, on the domain $\Omega  = \Omega_{\mathcal{C}}\cup \psi(\Omega_{\mathcal{C}})$ studied in Lemma \ref{lem:bicone}. 
Let us first define $H_0 \subset H^1_{\rm tr}(\Omega)$ as the set of all elements $u\in H^1_{\rm tr}(\Omega)$ such that ${\rm tr}(u) = 0$  on $]-1,3[\times\{1\}$ and  ${\rm tr}(u) = 0$  on $]-1,3[\times\{-1\}$,  $\mu_\theta$-a.e. for all $\theta\in\mathbb{S}^{d-1}$. The Poincar\'e inequality 
\eqref{eq:pointheta} used with $\theta = (0,\pm 1)$ suffices to prove that $H_0$ is a Hilbert space with the scalar product $\int_\Omega \nabla u(x)\cdot\nabla v(x){\rm d}x$. We now study the two following problems.

\medskip

{\bf Problem in $\widetilde{H}^1(\Omega)$:} Find $u\in \widetilde{H}^1(\Omega)$ such that
\[ \begin{cases}
    {\rm tr}(u) = 1 \text{ on } ]-1, 3[ \ \times \{1\} \text{ and } {\rm tr}(u) = -1 \text{ on } ]-1, 3[ \ \times \{-1\} \ \mu_\theta\text{-a.e. for all } \theta\in\mathbb{S}^{d-1}, \vspace{0.2cm} \\
    \displaystyle \forall v\in H_0\cap \widetilde{H}^1(\Omega), \int_\Omega \nabla u(x)\cdot\nabla v(x){\rm d}x = 0.
\end{cases} \]
Then the solution of this problem is the function defined by $u(x,y) = y$ a.e. in $\Omega$. \pagebreak

{\bf Problem in $H^1_{\rm tr}(\Omega)$:} Find $u\in H^1_{\rm tr}(\Omega)$ such that 
\[ \begin{cases}
    {\rm tr}(u) = 1 \text{ on } ]-1, 3[ \ \times \{1\} \text{ and } {\rm tr}(u) = -1 \text{ on } ]-1, 3[ \ \times \{-1\} \ \mu_\theta\text{-a.e. for all } \theta\in\mathbb{S}^{d-1}, \vspace{0.2cm} \\
    \displaystyle \forall v\in H_0, \int_\Omega \nabla u(x)\cdot\nabla v(x){\rm d}x = 0.
\end{cases} \]
The solution of this problem is the function defined by $u(x,y) = 1$ a.e. in $\Omega$ for $y>0$ and  $-1$ for $y<0$.

\subsection{The case of Lipschitz domains}

We recall that, in the case of Lipschitz domains, the relation $\widetilde{H}^1(\Omega) = H^1(\Omega)$ holds, and therefore  $H^1_{\rm tr}(\Omega) = H^1(\Omega)$. The trace in the standard sense is obtained by passing to the limit in $L^2(\partial\Omega, \mathcal{H}^{d-1})$ on the restriction of $C^1$ functions on the boundary, whereas the omnidirectional trace is obtained by passing to the limit in $L^2(\partial\Omega,(\mu_\theta)_{\theta\in\mathbb{S}^{d-1}})$. Although the spaces $L^2(\partial\Omega, \mathcal{H}^{d-1})$ and $L^2(\partial\Omega,(\mu_\theta)_{\theta\in\mathbb{S}^{d-1}})$ do not have the same construction, the following lemma implies that the standard trace coincides with the omnidirectional trace in this case.

\begin{lemma}[The measure $\mu_\theta$ in the case of a Lipschitz domain]\label{exa:cun}
    Let $\Omega$ be a Lipschitz domain, denote by ${\bm n}(z)$ the unit outward normal to the boundary for $\mathcal{H}^{d-1}$-a.e. $z\in\partial\Omega$. Then 
    \[ \mu_\theta(z) = \ell_{\minus\theta}(z) \max(\theta\cdot {\bm n}(z),0){\mathcal{H}}^{d-1}(z). \]
\end{lemma}

\begin{proof}
    Let $\varphi\in C^0(\partial\Omega)$ be given. We define the function ${\bm u}:\Omega\to \mathbb{R}^d$ such that 
    \[ {\bm u}(x) = \big((x- \Phi_{\minus\theta}(x))\cdot\theta\big) \varphi(\Phi_{\theta}(x))\theta. \]
    Note that $u \in H(\mathrm{div}, \Omega)$ with $\mathrm{div}\, u(x) = \varphi(\Phi_\theta(x))$ (this follows by taking a basis with $\theta$ as the first vector). We also notice that the function ${\bm u}$ has a normal trace $\gamma{\bm u}\in L^\infty(\partial\Omega,\mathcal{H}^{d-1})$ which is such that $\gamma{\bm u}(z) = \varphi(z) \ell_{\minus\theta}(z) \max(\theta\cdot {\bm n}(z),0)$ for $\mathcal{H}^{d-1}$-a.e. $z\in\partial\Omega$. Then, applying the divergence theorem (also named the  Gauss-Green theorem \cite{pfeffer}), we get
     \[ \int_\Omega {\rm div}{\bm u}(x) {\ \rm d}x 
        = \int_\Omega\varphi(\Phi_\theta (x)) {\ \rm d}x 
        = \int_{\partial\Omega} \varphi(z) \ell_{\minus\theta}(z) \max(\theta\cdot {\bm n}(z),0) {\rm d}{\mathcal{H}}^{d-1}(z). \]
    The identification of the last term with $\int_{\partial\Omega} \varphi(z) {\rm d}\mu_\theta(z)$ concludes the proof.
\end{proof}

\subsection{The case of domains $\Omega$ such that $H_{\rm tr}^1(\Omega)\neq H^1(\Omega)$}\label{exa:cracks}

The presence of cracks in the domain prevents from the equality $H_{\rm tr}^1(\Omega)= H^1(\Omega)$, as shown by the two following examples.

\begin{example}(1D case where $H_{\rm tr}^1(\Omega)\neq H^1(\Omega)$){\bf .}\label{exa:fisund}
    Let us take the example $\Omega = \ ]0,1[ \ \cup \ ]1,2[$, and $u(x) = x$ for $x \in \ ]0,1[$, and  $u(x) = x-1$ for $x \in \ ]1,2[$. Then $\gamma_1 u(1) = 1$, $\gamma_1 u(2) = 1$, $\gamma_{\minus 1} u(0) = 0$, $\gamma_{\minus 1} u(1) = 0$. Hence in general, for $z\in \partial\Omega$, $\gamma_\theta u(z)$ is not constant with respect to $\theta$.
\end{example}

\begin{example}(2D case where $H_{\rm tr}^1(\Omega)\neq H^1(\Omega)$){\bf .}\label{exa:fisdeuxd}
    Let us consider $\Omega = (]0,1[\times ]-1,1[) \setminus (\{\frac 1 2\}\times[0,1])$. We define the function $u(x_1,x_2) = - x_2$ for $(x_1,x_2)\in ]0,\frac 1 2[\times ]0,1[ $, $u(x_1,x_2) = x_2$ for $(x_1,x_2)\in ]\frac 1 2,1[\times ]0,1[ $ and $u(x_1,x_2) = 0$ for $(x_1,x_2)\in ]0,1[\times ]-1,0[$. The domain is connected, unlike Example \ref{exa:fisund}, and $\gamma_{(1,0)} u( (\frac 1 2,s)) = -s$ and  $\gamma_{(-1,0)} u( (\frac 1 2,s)) = s$ for all $s\in]0,1[$.
\end{example}

\subsection{The case of a cuspidal domain}\label{exa:cuspidal}

Recall that cuspidal domains are not Lipschitz domains. Nevertheless, the presence of one cusp, even severe, does not prevent the equality 
 $H^1_{\rm tr}(\Omega) = H^1(\Omega)$. It is noteworthy that in this case, $H^{1/2}(\partial\Omega)$ is not a subset of $L^2(\partial\Omega, \mathcal{H}^{d-1})$.
    Let us assume that $\Omega\subset\mathbb{R}^2$ is the (non-Lipschitz) domain defined by (see Figure \ref{fig:triangle})
    \[ \Omega = \{(x_1,x_2)\in \mathbb{R} \times ]0,1[ \ : -x_2^3 < x_1 < x_2^3\}. \]
    
    \begin{figure}[!ht]
    \begin{center}

    \begin{tikzpicture}[scale=2]
    \fill[red!10] 
        plot[domain=0:1, variable=\y] ({-\y^3}, \y) -- 
        plot[domain=1:0, variable=\y] ({\y^3}, \y) -- cycle;

    \draw[thick, black] plot[domain=0:1, variable=\y] ({\y^3}, \y);
    \draw[thick, black] plot[domain=0:1, variable=\y] ({-\y^3}, \y);
    
    \draw[thick, black] (-1,1) -- (1,1);

    \draw[->] (-1.5,0) -- (1.5,0) node[right] {$x_1$};
    \draw[->] (0,-0.2) -- (0,1.5) node[above] {$x_2$};
\end{tikzpicture}
\end{center}
        \caption{Cuspidal domain $\Omega$. \label{fig:triangle}}
    \end{figure}
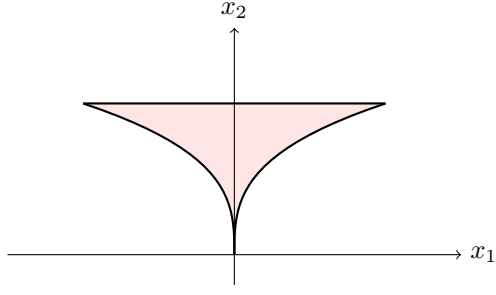

    Assume that $\theta = (1,0)$. Then we get for all $(x_1,x_2)\in\Omega$ that $\Phi_{\theta}(x_1,x_2) = (x_2^3,x_2)$, $\ell_{\theta}(x_1,x_2) = 2 x_2^3$. Considering the right side of the domain (in red in  Figure \ref{fig:triangle}), we have 
    \[ \theta\cdot {\bm n}(x_2^3,x_2) = \frac 1 {\sqrt{1 + 9 x_2^4}}. \] 
    Hence, for a.e. $x_2 \in ]0,1[$, 
    \[ \mu_{\theta}(x_2^3,x_2) = \frac {2 x_2^3} {\sqrt{1 + 9 x_2^4}}  \mathcal{H}^1(x_2^3,x_2), \]
    and $\mu_{\theta}$ is equal to $0$ on the remaining of the boundary (in black in Figure \ref{fig:triangle}). Then the function $u(x_1,x_2) = x_2^{-\alpha}$ belongs to $H^1(\Omega)\subset W_\theta(\Omega)$ for all $\alpha \in \ ]\frac 1 2,1[$ (indeed, the function $x_2^{-2\alpha-2} \times x_2^3$ is integrable on $]0,1[$), and a directional trace of $u$ in the sense of Definition \ref{def:dirtrace} is the function $g(x_1,x_2) = x_2^{-\alpha}$ on all edges. In this case, $g$ is also equal a.e. to the trace of $u$ as defined in Section \ref{sec:tracehun}. Note that $g \in \mathcal{L}^2(\partial\Omega, \mu_\theta)$ but $g \notin \mathcal{L}^2(\partial\Omega, \mathcal{H}^1)$ (the function $x_2^{-2\alpha}$ is not integrable on $]0,1[$).

\subsection{Cases  in which the Hausdorff measure $\mathcal{H}^{d-1}$ is locally infinite on $\partial\Omega$}\label{sec:cantor}

\begin{example}(complement of Cantor set){\bf .} \label{ex:cantor}
    For $\rho\in \ ]0,\frac 1 3]$, we construct the $\rho$-Cantor set, denoted by $\mathcal{C}_\rho\subset [0,1]$ by the following procedure.
    \begin{itemize}
        \item We define $a_1=0$ and $b_1 = 1$;
        \item For all $k\in\mathbb{N}$ $m=2^k,\ldots,2^{k+1}-1$, we define $c_m = \frac {a_m+b_m} 2 -\frac  {\rho^{k+1}} 2$ and $d_m = \frac {a_m+b_m} 2 +\frac  {\rho^{k+1}} 2$; we then define $a_{2m} = a_m$, $b_{2m} = c_m$, $a_{2m+1} = d_m$, $b_{2m+1} = b_m$;
    \end{itemize}
    Hence we have $c_1 = \frac {a_1+b_1} 2 - \frac  \rho 2$,  $d_1 = \frac {a_1+b_1} 2 + \frac \rho 2$, $a_2 = a_1$, $b_2 = c_1$, $a_3 = d_1$ and $b_3 = b_1$. We set
    \[ \Omega = \bigcup_{m\ge 1} ]c_m,d_m[. \]
    Then $\mathcal{C}_\rho = \partial\Omega = [0,1] \setminus \Omega$, since $\Omega$ is dense in $[0,1]$. If $\rho = 1/3$, we recover the well-known ``middle third Cantor set''. The Lebesgue measure of $\Omega$ is equal to $\frac {\rho} {1 - 2\rho}$, therefore the Lebesgue measure of $\mathcal{C}_\rho$ is equal to $\frac {1-3\rho} {1 - 2\rho}$. Since we are in 1D, the unit sphere is $\mathbb{S}^{d-1} = \{-1,1\}$. The measure $\mu_1$ (resp. $\mu_{\minus 1}$) is the discrete measure on points $d_m$ (resp. $c_m$) with weight $d_m-c_m$. This means that, for all $A \in \mathcal{B}(\partial \Omega)$,
    $$
    \mu_1(A)= \sum_{m \ge 1} (d_m-c_m) \chi_A(d_m)~\text{and}~\mu_{\minus 1}(A)=\sum_{m \ge 1} (d_m-c_m) \chi_A(c_m).
    $$
    The Hausdorff measure $\mathcal{H}^{0}$ (which is the counting measure) of $ \partial\Omega$ is infinite, the Hausdorff measure $\mathcal{H}^{1}$ (which coincide with the Lebesgue measure) of $ \partial\Omega$ is null for $\rho = 1/3$ and positive for $\rho < 1/3$. This examples cannot be handled in the framework of the approximative trace of \cite{ae2011dirtoneu,sauter2020uniq} since the measure $\mathcal{H}^{d-1}$ is locally infinite everywhere on $\partial\Omega$. Indeed, we have $\partial\Omega = \mathcal{C}_\rho$ and $\overline{\Omega} = [0,1]$.
\end{example}


\begin{figure}[!ht]
\begin{center}
    \begin{tikzpicture}[x=6cm, y=3cm] 
        
        \def\rho{0.25}
        \def\hOrig{1}
        \def\hBase{1}
        \colorlet{fillColor}{red!10}

        \expandafter\gdef\csname a1\endcsname{0}
        \expandafter\gdef\csname b1\endcsname{1}
        
        \fill[black] (0,0) rectangle (1,\hOrig+\hBase);

        \fill[fillColor] (0,0) rectangle (1,\hBase);

        
        \foreach \k in {0,1,2,3,4} {
            \pgfmathsetmacro{\startM}{int(2^\k)}
            \pgfmathsetmacro{\endM}{int(2^(\k+1)-1)}
            
            \foreach \m in {\startM,...,\endM} {
                \pgfmathsetmacro{\am}{\csname a\m\endcsname}
                \pgfmathsetmacro{\bm}{\csname b\m\endcsname}
                
                \pgfmathsetmacro{\holeWidth}{(\rho)^(\k+1)}
                \pgfmathsetmacro{\midpoint}{(\am + \bm)/2}
                \pgfmathsetmacro{\cm}{\midpoint - \holeWidth/2}
                \pgfmathsetmacro{\dm}{\midpoint + \holeWidth/2}
                
                \fill[fillColor] (\cm, \hBase) rectangle (\dm, \hOrig+\hBase);
                
                \pgfmathsetmacro{\idxL}{int(2*\m)}
                \pgfmathsetmacro{\idxR}{int(2*\m+1)}
                
                \expandafter\xdef\csname a\idxL\endcsname{\am}
                \expandafter\xdef\csname b\idxL\endcsname{\cm}
                \expandafter\xdef\csname a\idxR\endcsname{\dm}
                \expandafter\xdef\csname b\idxR\endcsname{\bm}
            }
        }

        
        \draw[thin, black] (0,0) rectangle (1, \hBase+\hOrig);
        
        \node[below, font=\tiny] at (0, 0) {0};
        \node[below, font=\tiny] at (1, 0) {1};

    \end{tikzpicture}
     \begin{tikzpicture}[x=6cm, y=3cm] 
        
        \def\rho{0.33333}
        \def\hOrig{1}
        \def\hBase{1}
        \colorlet{fillColor}{red!10}

        \expandafter\gdef\csname a1\endcsname{0}
        \expandafter\gdef\csname b1\endcsname{1}
        
        \fill[black] (0,0) rectangle (1,\hOrig+\hBase);

        \fill[fillColor] (0,0) rectangle (1,\hBase);

        
        \foreach \k in {0,1,2,3,4} {
            \pgfmathsetmacro{\startM}{int(2^\k)}
            \pgfmathsetmacro{\endM}{int(2^(\k+1)-1)}
            
            \foreach \m in {\startM,...,\endM} {
                \pgfmathsetmacro{\am}{\csname a\m\endcsname}
                \pgfmathsetmacro{\bm}{\csname b\m\endcsname}
                
                \pgfmathsetmacro{\holeWidth}{(\rho)^(\k+1)}
                \pgfmathsetmacro{\midpoint}{(\am + \bm)/2}
                \pgfmathsetmacro{\cm}{\midpoint - \holeWidth/2}
                \pgfmathsetmacro{\dm}{\midpoint + \holeWidth/2}
                
                \fill[fillColor] (\cm, \hBase) rectangle (\dm, \hOrig+\hBase);
                
                \pgfmathsetmacro{\idxL}{int(2*\m)}
                \pgfmathsetmacro{\idxR}{int(2*\m+1)}
                
                \expandafter\xdef\csname a\idxL\endcsname{\am}
                \expandafter\xdef\csname b\idxL\endcsname{\cm}
                \expandafter\xdef\csname a\idxR\endcsname{\dm}
                \expandafter\xdef\csname b\idxR\endcsname{\bm}
            }
        }

        
        \draw[thin, black] (0,0) rectangle (1, \hBase+\hOrig);
        
        \node[below, font=\tiny] at (0, 0) {0};
        \node[below, font=\tiny] at (1, 0) {1};

    \end{tikzpicture}
    \end{center}
    \caption{Example \ref{exa:bicantor}. Left: $\rho = \frac 1 4$. Right: $\rho = \frac 1 3$. \label{fig:bicantor}}
\end{figure}
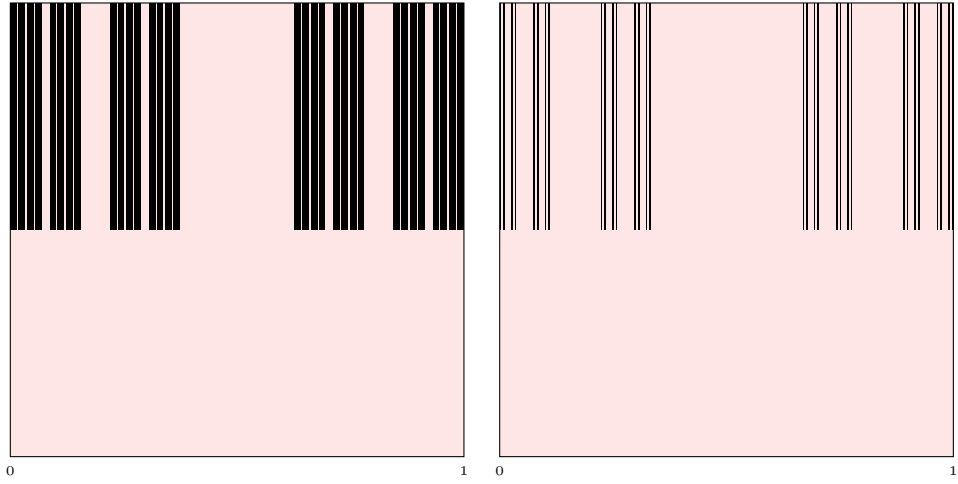

\begin{example}\label{exa:bicantor}
    The domains $\Omega$ is shown in Figure \ref{fig:bicantor} in the cases $\rho = \frac 1 4$ and $\rho = \frac 1 3$. We denote, for a given $\rho\in \ ]0,\frac 1 3]$, by
    \[ \Omega = \Big( ([0,1]\setminus \mathcal{C}_\rho )\times ]-1,1[ \Big)\cup \Big(]0,1[\times ]-1,0[\Big). \]
    Then the boundary of $\Omega$ contains $\mathcal{C}_\rho \times [0,1[$, which has infinite local Hausdorff measure $\mathcal{H}^{d-1}$ for $\rho < \frac 1 3$, since it contains the set  $\mathcal{C}_\rho \times [0,1]$. But the set $\partial_{\mathbb{S}}\Omega  $ is given by the union of all the following sets:
    \begin{itemize}
        \item the boundary of $]0,1[\times ]-1,1[$,
        \item the sets of the form $\{c_m\}\times [0,1]$ and $\{d_m\}\times [0,1]$,
        \item the set $\mathcal{C}_\rho \times \{0\}$.
    \end{itemize} \pagebreak
    Then $\Omega$ is connected and $H_{\rm tr}^1(\Omega) = H^1(\Omega)$, applying Lemma \ref{lem:sufcondmono} below with $\Omega_0 = \ ]0,1[\times ]-1,0[$ and $\Omega_m = ]c_m,d_m[\times ]0,1[$.
\end{example}

The next lemma gives a simple sufficient condition leading to $H_{\rm tr}^1(\Omega) = H^1(\Omega)$.

\begin{lemma}\label{lem:sufcondmono}
    Let $(\Omega_i)_{i\in I}$ be a countable family of open subsets of $\Omega$, such that
    \begin{itemize}
        \item for all $i\in I$, $H_{\rm tr}^1(\Omega_i) = H^1(\Omega_i)$,
        \item for all $\theta\in \mathbb{S}^{d-1}$, $\partial_\theta\Omega\setminus \bigcup_{i\in I} \partial_\theta\Omega_i$ is $\mu_\theta$-negligible,
        \item for all $i\in I$, $\partial_{\mathbb{S}}\Omega  \cap \partial_{\mathbb{S}}\Omega _i \cap \bigcup_{j\neq i} \partial_{\mathbb{S}}\Omega _j $ is $\mu_\theta$-negligible for all $\theta\in \mathbb{S}^{d-1}$ (see \eqref{eq:defdomegatilde} for the definition of $\partial_{\mathbb{S}}\Omega $). 
    \end{itemize}
    Then $H_{\rm tr}^1(\Omega) = H^1(\Omega)$.  
\end{lemma}

\begin{proof}
    Let $u\in H^1(\Omega)$ and let $u_i\in H^1(\Omega_i)$ be the restriction of $u$ to $\Omega_i$ for all $i\in I$. We denote by $g_i$ a representative of ${\rm tr }(u_i)$, the trace of $u_i$ on $\partial\Omega_i$, defined by Definition \ref{def:tracemono}. By Lemma \ref{lem:included}, we have that, for all $\theta\in \mathbb{S}^{d-1}$,  $g_i(z) = \gamma_\theta u_i(z) = \gamma_\theta u(z)$ for $\mu_{\theta}$-a.e. $z\in \partial_\theta\Omega_i\cap\partial_\theta\Omega$. \medskip
 
    Let us define the function $g:\partial\Omega\to\mathbb{R}$ by $g(z) = g_i(z)$ for all $z\in (\partial_{\mathbb{S}}\Omega \cap \partial_{\mathbb{S}}\Omega _i) \setminus \bigcup_{j\neq i} \partial_{\mathbb{S}}\Omega _j $ and all $i\in I$, and $g(z) = 0$ for all other $z\in\partial\Omega$. We then have, for all $\theta\in\mathbb{S}^{d-1}$, $g(z) = g_i(z) = \gamma_\theta u(z)$ for $\mu_\theta$-a.e. $z\in \partial_\theta\Omega_i \cap \partial_\theta\Omega \subset \partial_{\mathbb{S}}\Omega \cap \partial_{\mathbb{S}}\Omega _i $. Since the complement in $\partial_\theta\Omega$ of $\bigcup_{i\in I}(\partial_\theta\Omega_i\cap\partial_\theta\Omega)$ is $\mu_\theta$-negligible, we obtain that, for  all $\theta\in \mathbb{S}^{d-1}$, $g(z) = \gamma_\theta u(z)$ for  $\mu_{\theta}$-a.e. $z\in \partial_{\theta}\Omega$. This implies that the equivalence class of $g$ in $L^2(\partial\Omega,(\mu_\theta)_{\theta\in\mathbb{S}^{d-1}})$ is the trace of $u$ in the sense of Definition \ref{def:tracemono}.
\end{proof}

\section{Conclusion and open problems}
The following inclusion relations 
 \[ H^1_0(\Omega)  \subset H_{\rm tr,0}^1(\Omega)\subset \widetilde{H}^1(\Omega) \subset H^1_{\rm tr}(\Omega)\subset H^1(\Omega) \]
can be strict in situations which are explored in this paper. We do not yet have at this point general conditions on the domain shape leading to the equality of some of these sets. For example, if $\Omega$ is equal to the interior of its closure, additional properties are expected.
The extension of the results of the present paper to the spaces $W^{1,p}(\Omega)$ for $p \in [1, +\infty)$ remains to be studied.
Nevertheless, other questions remain open, for example:
\begin{itemize}
    \item Can we establish the connection between the collection of measures $(\mu_\theta)_{\theta\in\mathbb{S}^{d-1}}$ and the harmonic measures on $\partial\Omega$ in the general case, in the spirit of \cite{dahlberg}?
    
    \item How many values can the directional traces of elements of $H^1(\Omega)$ in all directions $\theta\in\mathbb{S}^{d-1}$ take almost everywhere?
\end{itemize}
These questions justify further investigations.

\appendix

\section{Properties of functions in $H^1(]\alpha, \beta[)$}\label{ap:hund}

\begin{lemma}\label{lem:fhun}
    Let $\alpha<\beta$ be two reals and let $f \in H^1(]\alpha,\beta[) \cap C^0([\alpha,\beta])$ be given. Denoting by $Df$ the weak derivative of $f$, the following properties hold:
    \begin{equation}\label{eq:valbeta}
        f(\beta) = \frac 1 {\beta - \alpha} \int_\alpha^\beta 
            \big(f(t) + (t-\alpha) Df(t)\big){\rm d}t,
    \end{equation}
    
    \begin{equation}\label{eq:majfhun}
        \max(f(\beta)^2, f(\alpha)^2) 
            \le \Vert f\Vert_{H^1(]\alpha,\beta[)}^2 
            \frac{2\max(1, (\beta - \alpha)^2)}{\beta - \alpha},
    \end{equation}
    
    \begin{equation}\label{eq:majsumfhun}
        (f(\beta)+f(\alpha))^2 
            \le \Vert f\Vert_{H^1(]\alpha,\beta[)}^2 
            \frac{4\max(1, (\beta - \alpha)^2)}{\beta - \alpha},
    \end{equation}
    and
    \begin{equation}\label{eq:majdiffhun}
        \Big(\frac{f(\beta) - f(\alpha)} {\beta - \alpha}\Big)^2 
        \le \Vert f\Vert_{H^1(]\alpha,\beta[)}^2\frac{1}{\beta - \alpha}.
    \end{equation}
    
    Moreover, if $f(\beta)=0$, then
    \begin{equation}\label{eq:poinfhun}
        \Vert f\Vert_{L^2(]\alpha,\beta[)} 
        \le (\beta - \alpha) \Vert D f\Vert_{L^2(]\alpha,\beta[)}.
    \end{equation}
\end{lemma}

\begin{proof}
    Let $s \in [\alpha, \beta]$. For $t \in [\alpha, \beta]$, let us define
    \[ \psi(s,t) = \begin{cases}
        t-\alpha & \text{ if } t < s, \\
        t-\beta & \text{ if } t > s.
    \end{cases}\]
    Then, using an integration-by-parts formula in $H^1(]\alpha, \beta[)$, we get
    \begin{align*}
        \int_\alpha^\beta \big( f(t) + \psi(s,t) Df(t) \big) {\ \rm d}t
        & = \int_\alpha^\beta f(t) {\ \rm d}t
            + \int_\alpha^s (t-\alpha) Df(t) {\ \rm d}t
            + \int_s^\beta (t-\beta) Df(t) {\ \rm d}t \\
        & = \int_\alpha^\beta f(t) {\ \rm d}t
            + (s-\alpha) f(s) - \int_\alpha^s f(t) {\ \rm d}t
            - (s-\beta) f(s) - \int_s^\beta f(t) {\ \rm d}t \\
        & = (\beta - \alpha)f(s).
    \end{align*}

    Applying this formula to $s = \beta$ yields \eqref{eq:valbeta}. Using the Cauchy-Schwarz inequality, we obtain
    \begin{align*}
        f(s)^2 & = \frac{1}{(\beta-\alpha)^2} \left(\int_\alpha^\beta [f(t) + \psi(s,t) Df(t)] \cdot 1 {\ \rm d}t\right)^2 \\
        & \le \frac{1}{(\beta-\alpha)^2} \left(\int_\alpha^\beta [f(t) + \psi(s,t) Df(t)]^2 {\ \rm d}t\right)  \left(\int_\alpha^\beta {\rm d}t\right) \\
        & \le \frac{2}{\beta-\alpha} \left(\int_\alpha^\beta [f(t)^2 + \psi(s,t)^2 Df(t)^2] {\ \rm d}t\right) \\
        & \le \frac{2}{\beta-\alpha} \left(\int_\alpha^\beta \max(1, (\beta-\alpha)^2)[f(t)^2 + Df(t)^2] {\ \rm d}t\right) \\
        & = \frac{2 \max(1, (\beta-\alpha)^2)}{\beta-\alpha} \Vert f \Vert_{H^1(]\alpha, \beta[)}^2.
    \end{align*}
    Therefore, we get \eqref{eq:majfhun}.  Adding \eqref{eq:valbeta} and
    \[
    f(\alpha) = \frac 1 {\beta - \alpha} \int_\alpha^\beta 
            \big(f(t) + (t-\beta) Df(t)\big){\rm d}t,
            \]
    the  Cauchy-Schwarz inequality leads to \eqref{eq:majsumfhun}. Using once again the Cauchy-Schwarz inequality,
    \begin{align*}
        [f(\beta) - f(\alpha)]^2 = \left[\int_\alpha^\beta Df(t) {\ \rm d}t\right]^2
        & \le \left(\int_\alpha^\beta [Df(t)]^2 {\ \rm d} t\right)\left(\int_\alpha^\beta {\rm d} t\right) \\
        & \le (\beta - \alpha) \Vert f \Vert_{H^1(]\alpha, \beta[)}^2,
    \end{align*}
    hence \eqref{eq:majdiffhun} is proved. Finally, if $f(\beta) = 0$, we have for all $s \in [\alpha,\beta]$,
    \[ f(s) = -\int_s^\beta Df(t) {\ \rm d}t, \]
    which yields
    \[ f(s)^2 \le (\beta - s)\int_s^\beta Df(t)^2{\ \rm d}t \le (\beta - \alpha) \Vert D f\Vert_{L^2(]\alpha,\beta[)}^2. \]
    Integrating the preceding inequality with respect to $s\in[\alpha,\beta]$ provides \eqref{eq:poinfhun}.
\end{proof}

\section{Properties of $\partial_\theta \Omega$, $\Phi_\theta$ and $\mu_\theta$}\label{sec:propdirectionalmeasures}

\begin{lemma}\label{lem:partialthetaboundary}
    Let $z \in \partial_\theta \Omega$. Then  $\omega_\theta(\mathcal{P}_\theta(z))$ is a non-empty open subset of $\mathbb{R}$ and there exists an open interval $]\alpha,\beta[ \ \in \mathcal{I}_{\theta}(\mathcal{P}_\theta(z))$ such that $z = \mathcal{P}_\theta(z) + \beta \theta$. Therefore
    \[ \partial_\theta\Omega := \{z \in \partial\Omega : \exists r>0,\ \forall t\in ]0,r[, \ z-t\theta\in\Omega\} =\{ \beta\theta+y : \ y\in \mathcal{P}_\theta(\Omega), \ ]\alpha,\beta[ \ \in \mathcal{I}_{\theta}(y)\}, \]
    and $\mathcal{P}_\theta(\Omega) = \mathcal{P}_\theta(\partial_\theta\Omega)$.
\end{lemma}

\begin{proof}
    Let $x\in\Omega$. Then, letting $y = \mathcal{P}_\theta(x)$, there exist $]\alpha,\beta[\ \in \mathcal{I}_{\theta}(y)$ and $s \in \ ]\alpha,\beta[ \ \subset \omega_\theta(y)$ such that $x = s\theta+y$. Then $z := \Phi_\theta(x) = \beta\theta+y \in \partial_\theta\Omega$ and $\mathcal{P}_\theta(z) = y$.
    We then notice that, letting $r=\beta-\alpha$, we have $x = z-t\theta$ with $t = \beta-s\in \ ]0,r[$.
\end{proof}

\begin{lemma}\label{lem:closomegatilde}
    The closure and the boundary of the set $\partial_{\mathbb{S}}\Omega $ defined in \eqref{eq:defdomegatilde} are equal to $\partial\Omega$.
\end{lemma}

\begin{proof}
    Let $z \in \partial\Omega$ and $(x_n)_{n \ge 0} \subset \Omega$ such that $x_n \to z$. For all $n \in \mathbb N$, denote
    \[ \theta_n = \frac{z-x_n}{|z-x_n|} \in \mathbb{S}^{d-1}, \qquad
        z_n = \Phi_{\theta_n}(x_n) = x_n + \delta_{\theta_n}(x_n) \frac{z-x_n}{|z-x_n|} \in \partial_{\theta_n}\Omega \subset \partial_{\mathbb{S}}\Omega . \]
    Then, $0 < \delta_{\theta_n}(x_n) = |z_n-x_n| \le |z-x_n|$ (because $z$ is in the half-line with initial point $x_n$ along the direction $\theta_n$, and $z \not\in \Omega$). Hence,
    \[ |z-z_n|
        = \left|z - x_n - \delta_{\theta_n}(x_n) \frac{z-x_n}{|z-x_n|}\right|
        = |z-x_n| \left|1 - \frac{\delta_{\theta_n}(x_n)}{|z-x_n|}\right|
        \le |z-x_n|. \]
    Therefore we get that $z_n \to z$, so $z$ is in the closure of $\partial_{\mathbb{S}}\Omega $. Since the interior of $\partial \Omega$ is empty, the boundary of $\partial_{\mathbb{S}}\Omega $ equals its closure, that is $\partial \Omega$.
\end{proof}

\begin{lemma}\label{lem:measr}
    For all $\theta\in \mathbb{S}^{d-1}$, the function $\delta_{\theta} : \Omega\to{\mathbb{R}}$  is lower semicontinuous (that is the set $\{ x \in \Omega~\text{such that}~\delta_\theta(x) >
\lambda \}$ is open in $\Omega$ for any $\lambda  \in \mathbb{R}$) and therefore measurable.
\end{lemma}

\begin{proof}
    Let $x \in \Omega$ and $\lambda < \delta_\theta(x)$. Assume that for all $n \in \mathbb N$ such that $B(x, \frac{1}{n}) \subset \Omega$, there exists $x_n \in B(x, \frac{1}{n})$ satisfying $\delta_\theta(x_n)\in]0,\lambda]$. We have $x_n + \delta_\theta(x_n) \theta \not\in \Omega$. Up to a subsequence, we can assume that $(\delta_\theta(x_n))_n$ converges to some $t \in [0, \lambda]$. Moreover, $x_n \to x$ and $\Omega^c$ is closed, so $x + t\theta \not\in \Omega$. Thus, $t \le \lambda < \delta_\theta(x)$, a contradiction.
\end{proof}

\begin{remark}
    The preceding proof could be extended to prove that the function defined on $\mathbb{S}^{d-1} \times \Omega$ by $(\theta,x)\mapsto \delta_{\theta}(x)$ is lower semicontinuous.
\end{remark}

\begin{lemma}\label{lem:dthetaomborel}
 Let $\theta \in \mathbb{S}^{d-1}$. Then the set $\partial_\theta \Omega$ is an element of $\mathcal{B}(\partial \Omega)$.
\end{lemma}

\begin{proof}
    Define 
    \[ A = \bigcup_{ r \in ]0,+\infty[~} \bigcap_{~t \in ]0,r[} (\Omega + t \theta). \]
    Using Lemma \ref{lem:partialthetaboundary} we have $\partial_\theta \Omega = A \cap \partial \Omega$. Let us prove that $A$ is an element of $\mathcal{B}(\mathbb{R}^d)$. We remark that, since $\mathbb{Q}$ is dense in $\mathbb{R}$, we may write
    $$
    A = \bigcup_{q\in\mathbb{Q},q> 0~} \bigcap_{~t \in ]0,q[} (\Omega + t \theta )
    $$
    Let us prove that, for any $q\in\mathbb{Q},q> 0$, the set $ \bigcap_{t \in ]0,q[} (\Omega + t \theta )$ is an element of $\mathcal{B}(\mathbb{R}^d)$.
    We have 
    $$
    \bigcap_{t \in ]0,q[} (\Omega + t \theta ) = \bigcap_{n\in\mathbb{N},n \ge \frac 2 q~} \bigcap_{~ t \in [ \frac{1}{n},q - \frac{1}{n}] } ( \Omega  + t \theta).
    $$
    Let us show that for any $(a,b) \in \mathbb{R}^2$ such that $a<b$ the set $\bigcap_{t \in [a,b]} (\Omega +t \theta)$ is open. We have
    $$
    \bigcap_{t \in [a,b]} (\Omega +t \theta) = \{ z \in \mathbb{R}^d : \forall t \in [a,b], \ \text{dist}(z-t\theta,\Omega^c) > 0\}.
    $$
    Using the fact that $t \to \text{dist}(z-t\theta,\Omega^c)$ is continuous on $[a,b]$, we obtain
    $$
    \bigcap_{t \in [a,b]} (\Omega +t \theta) = \left\{ z \in \mathbb{R}^d : \min_{t\in [a,b]} \text{dist}(z-t\theta,\Omega^c) > 0\right\}.
    $$
    Using the fact that $z \to \min_{t\in [a,b]} \text{dist}(z-t\theta,\Omega^c)$ is continuous, we obtain that $\bigcap_{t \in [a,b]} (\Omega +t \theta)$ is an open subset  of $\mathbb{R}^d$. This implies that $ \bigcap_{t \in ]0,q[} (\Omega + t \theta )$ is an element of $\mathcal{B}(\mathbb{R}^d)$ and concludes the proof.
\end{proof}

\begin{lemma}\label{lem:negl}
    Let $\theta \in \mathbb{S}^{d-1}$. Let $A$ be a non-empty element of $ {\mathcal B}(\partial_\theta\Omega)$, where ${\mathcal B}(\partial_\theta\Omega)$ is the family of all Borel sets of $\partial_\theta\Omega\subset \partial\Omega$. Then $\mu_\theta(A)>0$ if and only if $\lambda^{d-1}(\mathcal{P}_\theta(A))>0$.
\end{lemma}

\begin{proof}
    Let $A$ be a non-empty element of ${\mathcal B}(\partial_\theta\Omega)$. Since Lemma \ref{lem:dthetaomborel} proves that $\partial_\theta\Omega\in \mathcal{B}(\partial \Omega)$, we get that $\chi_A\circ \Phi_\theta$ is measurable.
    Recall that $\Omega = \{y + s\theta : y \in H_\theta, s \in \omega_\theta(y)\}$. Hence,
    \[ \mu_\theta(A) = \int_\Omega \chi_A(\Phi_\theta(x)) {\ \rm d}x
        = \int_{H_\theta} \left(\int_{\omega_\theta(y)} \chi_A(\Phi_\theta(y+s\theta)) {\ \rm d}s \right) {\rm d}y. \]

    Using  the fact that $H_\theta = \mathcal{P}_\theta(A) \cup H_\theta \setminus \mathcal{P}_\theta(A)$ we obtain
    \begin{multline*} \mu_\theta(A) 
        = \int_{H_\theta} \left(\int_{\omega_\theta(y)} \chi_A(\Phi_\theta(y+s\theta)) {\ \rm d}s \right) {\rm d}y \\
       =  \int_{H_\theta \setminus \mathcal{P}_\theta(A)} \left(\int_{\omega_\theta(y)} \chi_A(\Phi_\theta(y+s\theta)) {\ \rm d}s \right) {\rm d}y + \int_{\mathcal{P}_\theta(A)} \left(\int_{\omega_\theta(y)} \chi_A(\Phi_\theta(y+s\theta)) {\ \rm d}s \right) {\rm d}y.
        \end{multline*}
    Note that $\mathcal{P}_\theta(A)$ is not necessarily an element of the Borel algebra. Nevertheless, it is possible to prove that $\mathcal{P}_\theta(A)$ is an element of the Lebesgue algebra (see \cite{Cohn2013measure}) and the above integrations are done with respect to the completed measure.
    We consider two different cases: 
    \begin{itemize}
        \item Assume that $y \in H_\theta \setminus \mathcal P_\theta(A)$ satisfies $\omega_\theta(y) \neq \varnothing$. Let $s \in \omega_\theta(y)$. Then
        \[ \Phi_\theta(y+s\theta) = y+s\theta+\delta_\theta(y+s\theta)\theta
            \implies \mathcal P_\theta(\Phi_\theta(y+s\theta)) = y \in H_\theta \setminus \mathcal P_\theta(A). \]
        Thus, $\Phi_\theta(y+s\theta) \not\in A$, so $\chi_A(\Phi_\theta(y+s\theta)) = 0$, and
        \[ \int_\Omega \chi_A(\Phi_\theta(x)) {\ \rm d}x
        = \int_{\mathcal P_\theta(A)} \left(\int_{\omega_\theta(y)} \chi_A(\Phi_\theta(y+s\theta)) {\ \rm d}s\right) {\rm d}y. \medskip\]

        \item If $y \in \mathcal P_\theta(A)$, then there exists $z = z(y) \in A$ such that $y = z - \langle z, \theta \rangle \theta$, so $z = y + \langle z, \theta\rangle \theta$. Since $A \subset \partial_\theta \Omega$, there exists $r > 0$ such that for all $t \in \ ]0, r[$, $z - t\theta \in \Omega$. Hence, if we set $s = \langle z, \theta \rangle - t$, we obtain
        \[ s \in ]\langle z, \theta\rangle - r, \langle z, \theta\rangle[
            \implies \Phi_\theta(y+s\theta) = \Phi_\theta(z-t\theta) = z. \]
        We deduce that $J = \ ]\langle z,\theta \rangle - r, \langle z, \theta \rangle[ \ \subset \omega_\theta(y)$, and
        \[ \int_{\omega_\theta(y)} \chi_A(\Phi_\theta(y+s\theta)) {\ \rm d}s
            \ge \int_J \chi_A(\Phi_\theta(y+s\theta)) {\ \rm d}s
            = \int_J \chi_A(z) {\ \rm d}s
            = \int_J {\rm d}s = \lambda(J) = r > 0. \]
    \end{itemize}
    Finally,
    \[
\mu_\theta(A) =  \int_{\mathcal{P}_\theta(A)} \left( \int_{\omega_\theta(y)} \chi_{A}( \Phi_\theta(s\theta+y)) {\ \rm d}s \right) {\rm d}y, \ \text{with} \ \int_{\omega_\theta(y)} \chi_{A}(\Phi_\theta(s\theta+y)) {\ \rm d}s > 0~\text{for all}~y \in \mathcal{P}_\theta(A). \]    
    so $\mu_\theta(A) = 0 \iff \lambda^{d-1}(\mathcal P_\theta(A)) = 0$, and this concludes the proof.
\end{proof}

We have the following corollary.

\begin{corollary}\label{cor:negl2}
    We have the following properties.
    \begin{enumerate}
        \item The measurable set $\partial\Omega\setminus \partial_\theta\Omega$ is $\mu_\theta$-negligible, which means that $\mu_\theta(\partial\Omega \setminus \partial_\theta\Omega)=0$.
        \item Let $A \subset \partial \Omega$ be non-empty. Then $A$ is $\mu_\theta$-negligible if and only if $\mathcal{P}_\theta (A \cap \partial_\theta \Omega)$ is $\lambda^{d-1}$-negligible.
    \end{enumerate}
\end{corollary}

\begin{proof} \;
    \begin{enumerate}
        \item We have $\partial\Omega \setminus \partial_\theta\Omega = \{z \in \partial \Omega : \forall r > 0, \ \exists t\in (0,r), \ z-t\theta \notin \Omega\}$. Therefore, by definition of $\Phi_\theta$, the set $\Phi_\theta(\Omega) \cap ( \partial\Omega\setminus \partial_\theta\Omega) $ is empty. This implies that $\mu_\theta(\partial\Omega\setminus\partial_\theta\Omega)=0$.

        \item If $A$ is $\mu_\theta$-negligible, then there exists $N \in \mathcal{B}(\partial \Omega)$ such that $A\subset N$ and $\mu_\theta(N)=0$. Then $A \cap \partial_\theta \Omega$ is a subset of $N \cap \partial_\theta \Omega \in {\mathcal B}(\partial_\theta\Omega)$ and $\mu_\theta(N \cap \partial_\theta \Omega) = 0$. Hence, $\mathcal{P}_\theta (A \cap \partial_\theta \Omega) \subset \mathcal{P}_\theta (N \cap \partial_\theta \Omega)$ and $\lambda^{d-1} (\mathcal{P}_\theta (N \cap \partial_\theta \Omega)) = 0$. Applying Lemma \ref{lem:negl}, $\mathcal{P}_\theta (A \cap \partial_\theta \Omega)$ is $\lambda^{d-1}$-negligible. \medskip

        Conversely, if $\mathcal{P}_\theta (A \cap \partial_\theta \Omega)$ is $\lambda^{d-1}$-negligible, then there exists $N \in \mathcal{B}(H_\theta) $ such that $\mathcal{P}_\theta (A \cap \partial_\theta \Omega)$ is a subset of $N$ and $\lambda^{d-1}(N)=0$. Hence, we obtain $A \cap \partial_\theta \Omega \subset \mathcal{P}_\theta^{-1}(N) \cap \partial_\theta \Omega \in \mathcal{B}(\partial_\theta \Omega)$. Moreover, $\mathcal{P}_\theta (\mathcal{P}_\theta^{-1}(N) \cap \partial_\theta \Omega) \subset N$, so applying Lemma \ref{lem:negl}, we obtain that $\mu_\theta(\mathcal{P}_\theta^{-1}(N) \cap \partial_\theta \Omega) = 0$. Finally, $A \cap \partial_\theta \Omega$ is $\mu_\theta$-negligible. To conclude, we need to note that 
        \[ A = (A \cap (\partial\Omega \setminus \partial_\theta\Omega)) \cup (A \cap \partial_\theta \Omega), \]
        the first set being also $\mu_\theta$-negligible by the first item. \qedhere
    \end{enumerate}
\end{proof}

\begin{lemma}\label{lem:visavis}
    Let $\theta \in \mathbb{S}^{d-1}$, $z \in \partial_\theta\Omega$, $y = \mathcal P_\theta(z) = \mathcal P_{\minus\theta}(z)$ and $]\alpha, \beta[ \ \in \mathcal I_\theta(y)$ such that $z = y + \beta \theta$. Then $\widehat z = y + \alpha \theta \in \partial_{\minus\theta} \Omega$, and $\Sym : z \in \partial_\theta \Omega \mapsto \widehat z \in \partial_{\minus\theta} \Omega$ is bijective. Hence,
    \[ \forall A \in \partial \Omega, \
        \mu_{\minus\theta}(A) = \mu_\theta(\Sym^{-1}(A)). \]
\end{lemma}

\begin{proof}
    First, recall that $]\alpha, \beta[ \ \in \mathcal I_\theta(y)$ if and only if $]-\beta, -\alpha[ \ \in \mathcal I_{\minus\theta}(y)$. Using Lemma \ref{lem:partialthetaboundary}, we deduce that $\widehat z = y + (-\alpha)(-\theta) \in \partial_{\minus\theta}\Omega$. We now show that $\Sym$ is bijective:
    \begin{itemize}
        \item Let $z, z' \in \partial_\theta \Omega$ such that $\Sym(z) = \Sym(z')$. By setting $y = \mathcal P_\theta(z)$ and $y' = \mathcal P_\theta(z')$, there exist $]\alpha, \beta[ \ \in \mathcal I_\theta(y)$ and $]\alpha', \beta'[ \ \in \mathcal I_\theta(y')$ such that $z = y+\beta \theta$ et $z' = y'+\beta'\theta$. Thus,
        \[ \left(y+\alpha \theta = \Sym(z) = \Sym(z') = y'+\alpha'\theta\right)
            \implies y-y' = (\alpha'-\alpha)\theta. \]
        However, $y-y' \in H_\theta$, so $\alpha'-\alpha = \langle y-y', \theta \rangle = 0$. Finally, $\alpha = \alpha'$, so we obtain $y = y'$ and $\mathcal I_\theta(y) = \mathcal I_\theta(y')$, implying that $\beta = \beta'$. In other words, $z = z'$. \smallskip

        \item Let $w \in \partial_{\minus\theta}\Omega$. Denoting $\xi = \mathcal P_{\minus\theta}(w)$, there exists $]s,t[ \ \in \mathcal I_{\minus\theta}(\xi)$ such that $w = \xi - t\theta$. Since $]-t,-s[ \ \in \mathcal I_\theta(\xi)$, we have $z = \xi - s\theta \in \partial_\theta \Omega$ (cf. Lemma \ref{lem:partialthetaboundary}), and $w = \Sym(z)$. \qedhere
    \end{itemize}
\end{proof}

\section{Bounded integrability with respect to a family of measures}\label{sec:intfam}

\begin{definition}\label{def:lpfamilymeasures}
    Let $(E, \mathcal T)$ be a measurable set, where we denote by $\mathcal T$ a given $\sigma$-algebra on the set $E$. Let $\mathcal{F}$ be any non-empty set  and let $(m_\theta)_{\theta \in \mathcal{F}}$ be a family of measures on $(E, \mathcal T)$. 
    For $p \in [1, +\infty[$, denote by $\mathcal L^p(E, \mathcal T, (m_\theta)_{\theta \in \mathcal{F}})$ the set of all maps $f : E \to \mathbb R$ satisfying
    \[ \exists M > 0, \; \forall \theta \in \mathcal{F}, \; \exists f_\theta \in \mathcal L^p(E, \mathcal T, m_\theta), \; f = f_\theta \; m_\theta\text{-a.e. and } \int_E |f_\theta|^p {\ \rm d} m_\theta \le M. \]
    Consider the following equivalence relation: $f \sim g$ if for all $\theta \in \mathcal{F}$, $f = g$ $m_\theta$-a.e. We denote by $L^p(E, \mathcal T, (m_\theta)_{\theta \in \mathcal{F}})$ the set of all equivalence classes of $\mathcal L^p(E, \mathcal T, (m_\theta)_{\theta \in \mathbb{S}^{d-1}})$.
\end{definition}

\begin{remark}
    In the preceding definition, we do not require that the elements of $\mathcal{L}^p(E,T,(m_\theta)_{\theta\in\mathcal{F}})$ be measurable, but only equal $m_\theta$-a.e. to a measurable function for all $\theta\in\mathcal{F}$.
\end{remark}

\begin{theorem}\label{theo:lpbanach} Let $p \in [1, +\infty[$ be given. Then
    $L^p(E, \mathcal T, (m_\theta)_{\theta \in \mathcal{F}})$ is a Banach space, endowed with the norm defined by
    \[ \Vert f \Vert_p = \sup_{\theta \in \mathcal{F}} \left(\int_E |f|^p{\ \rm d} m_\theta\right)^{1/p} = \sup_{\theta \in \mathcal{F}} \Vert f\Vert_{L^p(E, \mathcal T, m_\theta)}. \]
\end{theorem}

\begin{proof}
    We follow the proofs of \cite[Th\'eor\`emes 4.47 and 4.49]{gh}.
    Let $(f_n)_{n \ge 0} \subset L^p(E, \mathcal T, (m_\theta)_{\theta \in \mathcal{F}})$ be a Cauchy sequence, that is
    \[ \forall \varepsilon > 0, \; \exists N(\varepsilon) \ge 0, \; \forall q,r \ge N(\varepsilon), \; \Vert f_q-f_r \Vert_{p} \le \varepsilon. \]
    Let $k \ge 0$. Consider $N_k \ge N(2^{-k})$ so that $N_{k+1} > N_k$, and set
    \[ g_k = f_{N_{k+1}} - f_{N_k}. \]
    By definition of $N_k$, we obtain
    \[ \Vert g_k \Vert_{p} 
        = \sup_{\theta \in \mathcal{F}} \Vert g_k \Vert_{L^p(E, \mathcal T, m_\theta)}
        \le 2^{-k}. \]
    Let $j \ge 0$, and set
    \[ G_j = \sum_{k=0}^j |g_k|. \]
    For all $\theta \in \mathcal{F}$, by the triangle inequality, we have $\Vert G_j \Vert_{L^p(E, \mathcal T, m_\theta)} \le 2$. Hence, by the monotone convergence theorem, since $(G_j)$ is an increasing map,
    \[ G = \sum_{k \ge 0} |g_k| \in L^p(E, \mathcal T, m_\theta), \; \text{ and } \; \Vert G \Vert_{L^p(E, \mathcal T, m_\theta)} \le 2. \]
    Thus, there exists $E_\theta \subset E$ such that $m_\theta(E \backslash E_\theta) = 0$ and for all $x \in E_\theta$, $G(x) < \infty$. Then we can define, for all $x \in E_\theta$,
    \[ f(x) = \begin{cases}
        f_{N_0}(x) + \sum_{k \ge 0} g_k(x) & \text{ if } x \in E_\theta \\
        0 & \text{ if } x \not\in E_\theta.
    \end{cases} \]
    Then for all $x \in E_\theta$,
    \[ f_{N_j}(x) = f_{N_0}(x) + \sum_{k=0}^{j-1} g_k(x)
        \xrightarrow[j \to \infty]{} f(x). \]
    In addition, for all $j \ge 0$ and $x \in E_\theta$,
    \[ |f_{N_j}(x)| = |f_{N_0}(x) + \sum_{k=0}^{j-1} g_k(x)|
        \le |f_{N_0}(x)| + |G_j(x)|
        \le |f_{N_0}(x)| + |G(x)|. \]
    Using the dominated convergence theorem, $f \in L^p(E, \mathcal T, m_\theta)$. Moreover,
    \[ \sup_{\theta \in \mathcal{F}} \Vert f \Vert_{L^p(E, \mathcal T, m_\theta)}
        \le \sup_{\theta \in \mathcal{F}} \Vert f_{N_0} \Vert_{L^p(E, \mathcal T, m_\theta)}
            + \sup_{\theta \in \mathcal{F}} \Vert G \Vert_{L^p(E, \mathcal T, m_\theta)}
        \le 2 + \sup_{\theta \in \mathcal{F}} \Vert f_{N_0} \Vert_{L^p(E, \mathcal T, m_\theta)}. \]
    Thus, $f \in L^p(E, \mathcal T, (m_\theta)_{\theta \in \mathcal{F}})$. To finish, let $n \ge 0$. Considering $\ell \ge 0$ so that $N_\ell \le n$, we get
    \[ \Vert f-f_n \Vert_p
        \le \Vert f - f_{N_\ell} \Vert_p + \Vert f_{N_\ell} - f_n \Vert_p
        \le \sum_{k \ge \ell} \Vert g_k \Vert_p + 2^{-\ell}
        \le 2 \times 2^{-\ell}. \]
    Hence, we conclude that
    \[ \Vert f-f_n \Vert_p \xrightarrow[n \to \infty]{} 0. \]
    We showed that $L^p(E, \mathcal T, (m_\theta)_{\theta \in \mathcal{F}})$ is a complete space: it is a Banach space.
\end{proof}

\section{Generalisation of the Cantor staircase function}\label{sec:escalierdiable}

In this section, for the sake of completeness, we give an elementary proof of the existence of a function which has the properties of the Cantor staircase function: monotonous, continuous and constant on all the open intervals of a given family. This proof does not use the notion of singular measures with respect to the Lebesgue measure.
\begin{lemma}\label{lem:escalierdiable}
    Let  two reals $\alpha < \beta$ be given. Let $I\subset \mathbb{N}$ be a non-empty set, let 
    $(a_i,b_i)_{i\in I}$ be such that for all $i\in I$, $\alpha < a_i < b_i < \beta$ and all the intervals $[a_i,b_i]$ are disjoint.  There exists a continuous and increasing map $f : [\alpha,\beta] \to [0,1]$, such that $f(\alpha) = 0$, $f(\beta) = 1$ and $f$ is constant on all $[a_i,b_i]$, $i \in I$.
\end{lemma}

\begin{proof}
    For any reals $c<d$, define $I_{c,d} = \{i\in I : \ ]a_i,b_i[ \ \subset [c,d]\}$. If $I_{c,d} \neq \varnothing$, since $\sum_{i\in I_{c,d}}(b_i - a_i) \le d-c$, we can choose an element $i_{c,d}\in I_{c,d}$ (non-necessarily unique) such that, for all $j\in I_{c,d}$, it holds  $b_j - a_j \le b_{i_{c,d}} - a_{i_{c,d}}$.\medskip

    \textbf{Step 1}: We construct the sequence $(F_p)_{p\in\mathbb N}$.
    \begin{itemize}
        \item Let $F_0 = \{ (\alpha,\beta,0)\}$.
        \item For $p \in \mathbb N$, let $F_p$ be given. 
        For all $\xi := (c,d,m)$ in $F_p$, if  $I_{c,d} = \varnothing$, define $F_{p+1}^\xi = \{\xi\}$. Otherwise, define $F_{p+1}^\xi = \{ (c,a_{i_{c,d}},m+1),(b_{i_{c,d}},d,m+1)\}$ (observe that, since $c\in \{\alpha\}\cup \{b_i,i\in I\}$ and $d\in \{\beta\}\cup \{a_i,i\in I\}$, we have $c<a_{i_{c,d}}<b_{i_{c,d}}<d$). 
        
        Define $F_{p+1} = \bigcup_{\xi\in F_p}F_{p+1}^\xi$.
    \end{itemize}

First remark that, for a given $p\in \mathbb N$, if there are two distinct elements $(c_1,d_1,m_1)$ and $(c_2,d_2,m_2)$ in $F_p$, then $[c_1,d_1]\cap [c_2,d_2] = \varnothing$. Besides, this construction implies that, for all $p\in \mathbb N$ and all $i\in I$, either there exists $(c,d,m) \in F_p$ such that $i\in I_{c,d}$ (and then $m=p$ since $m<p$ implies that $I_{c,d} = \varnothing$), or there exists $c,d,m_1,m_2$ such that $\{ (c,a_i,m_1),(b_i,d,m_2)\}\subset F_p$. 
    Assume that for some $i\in I$,  there exist for all $p\in\mathbb N$ an element $(c_p,d_p,p) \in F_p$  such that $i\in I_{c_p,d_p}$. This means that $I_{c_p,d_p}$ is not empty, and that $j_p:=i_{c_p,d_p} \neq i$. We then get that $b_i - a_i \le b_{j_p} - a_{j_p}$ for all $p\in\mathbb N$. This is impossible since all $j_p$ must be distinct and $\sum_{p\in\mathbb N}(b_{j_p} - a_{j_p})\le \beta - \alpha$. Therefore, for all $i\in I$, denote by $p_i\in\mathbb N$ the greatest $p\in\mathbb N$ such that there exists $(c_p,d_p,p) \in F_p$  such that $i\in I_{c_p,d_p}$. \medskip

    \textbf{Step 2}: Based on the observation that, for all $p \in \mathbb N$, 
    \[ \sum_{(c,d,m)\in F_p} 2^{-m} = 1, \] 
    we define the function $u_p$ for a.e. $x\in [\alpha,\beta]$ by
    \[ u_p(x) = \sum_{(c,d,m)\in F_p} \frac{1}{2^{m} (d - c)} \chi_{]c,d[}(x), \]
    where $\chi_{]c,d[}$ is the characteristic function of $]c,d[$.
    Then we have $\int_0^1 u_p(x)= 1$, and, for all $i\in I$, $u_p(x) = 0$ a.e. on $]a_i,b_i[$ if $p>p_i$. The functions $u_p$ are such that, if $(c,d,m)\in F_p$, then for all $q\ge p$, $\int_c^d u_q(x){\rm d}x = 2^{-m}$. Hence, we may set
    \[ f_p : t \in [\alpha,\beta] \mapsto \int_\alpha^t u_p(x) \; \mathrm dx. \]
    For all $p \in \mathbb N$, $f_p$ is a continuous map on $[\alpha,\beta]$, which is affine and increasing on all $[c,d]$ such that there exists $m$ with $(c,d,m)\in F_p$, and constant elsewhere. Moreover, let $(c,d,m)\in F_p$.
    \begin{itemize}
    \item By construction, if $(c,d,m)\in F_{p+1}$, then $f_{q}(t) = f_p(t)$ for all $t\in [c,d]$ and all $q\ge p$.
    \item Otherwise, if there exists $a,b$ such that $\{(c,a,m+1),(b,d,m+1)\}\subset F_{p+1}$, then $m=p$ and
     the function $f_{p+1}- f_p$, which is continuous and piecewise affine, satisfies
        \[ (f_{p+1}- f_p)(c) = 0, \qquad (f_{p+1}- f_p)(a) = \frac{a - c}{2^{p} (d - c)} - \frac{1}{2^{p+1}} \in \left[- \frac{1}{2^{p+1}}, \frac{1}{2^{p+1}}\right], \]   
        \[ (f_{p+1}- f_p)(d) = 0, \qquad (f_{p+1}- f_p)(b) =  \frac{1}{2^{p+1}} - \frac{d - b}{2^{p} (d - c)} \in \left[- \frac{1}{2^{p+1}}, \frac{1}{2^{p+1}}\right]. \]
        This yields that, for all $t\in [c ,d]$, $|(f_{p+1}- f_p)(t)|\le  \frac{1}{2^{p+1}}$.
    \end{itemize}

    Finally, $||f_{p+1} - f_p||_{\infty} \le 2^{-1-p}$, so $(f_p)$ is a uniformly Cauchy sequence on $[\alpha,\beta]$. Hence, $(f_p)$ converges uniformly on $[\alpha,\beta]$ to a continuous map $f$, such that $f$ is non-decreasing (so as the $f_p$), $f(\alpha) = 0$, $f(\beta) = 1$ and $f$ is constant on $\bigcup_{i\in I} ]a_i,b_i[$ since, for all $p> p_i$, $f_p$ is constant on $[a_i,b_i]$.
\end{proof}

\end{document}